\documentclass[12pt]{article}
\usepackage{mathptmx}
\usepackage[latin9]{inputenc}
\usepackage[a4paper]{geometry}
\usepackage{amsmath}
\usepackage{amssymb}
\usepackage{wasysym}
\usepackage{esint}

\makeatletter
\@ifundefined{date}{}{\date{}}
\usepackage{amsfonts}

\allowdisplaybreaks[4]

\numberwithin{equation}{section}

\newtheorem{theorem}{Theorem}[section]

\newtheorem{lemma}[theorem]{Lemma}\newtheorem{proposition}[theorem]{Proposition}\newtheorem{remark}[theorem]{Remark}\newenvironment{proof}[1][Proof]{\textbf{#1.} }{\ \rule{0.5em}{0.5em}}

\makeatother

\begin{document}

\title{Parabolic equations with singular divergence-free \\ drift vector
fields}

\author{Zhongmin Qian\thanks{Research supported partly by an ERC grant. Mathematical Institute,
University of Oxford, OX2 6GG, England. Email: qianz@maths.ox.ac.uk} \ and\ Guangyu Xi\thanks{Research supported partly by Doctoral Training Center, EPSRC. Mathematical
Institute, University of Oxford, OX2 6GG, England. Email: guangyu.xi@maths.ox.ac.uk}}
\maketitle
\begin{abstract}
In this paper, we study an elliptic operator in divergence-form but
not necessarily symmetric. In particular, our results can be applied
to elliptic operator $L=\nu\Delta+u(x,t)\cdot\nabla$, where $u(\cdot,t)$
is a time-dependent vector field in $\mathbb{R}^{n}$, which is divergence-free
in distribution sense, i.e. $\nabla\cdot u=0$. Suppose $u\in L_{t}^{\infty}(\textrm{BMO}_{x}^{-1})$.
We show the existence of the fundamental solution $\varGamma(x,t;\xi,\tau)$
of the parabolic operator $L-\partial_{t}$, and show that $\varGamma$
satisfies the Aronson estimate with a constant depending only on the
dimension $n$, the elliptic constant $\nu$ and the norm $\left\Vert u\right\Vert _{L^{\infty}(\textrm{BMO}^{-1})}$.
Therefore the existence and uniqueness of the parabolic equation $\left(L-\partial_{t}\right)v=0$
are established for initial data in $L^{2}$-space, and their regularity
is obtained too. In fact, we establish these results for a general
non-symmetric elliptic operator in divergence form. 
\end{abstract}
%%\mainmatter

\section{Introduction}

The analysis of the Navier-Stokes equations, which are non-linear
partial differential equations describing the motion of incompressible
fluids confined in certain spaces, has inspired the large portion
of the mathematical analysis of non-linear partial differential equations
(see for example \cite{Ladyzenskaja-book-1968,Lemarie-Rieusset-book-2016,Morrey-1966,Stein-book-1970}
and etc.) due to the fundamental work J. Leray \cite{Leray-1934}.
The Navier-Stokes equations are partial differential equations of
second-order 
\begin{equation}
\frac{\partial}{\partial t}u+u\cdot\nabla u=\nu\Delta u-\nabla p,\label{eq:ns01}
\end{equation}
\begin{equation}
\nabla\cdot u=0,\label{eq:ns02}
\end{equation}
subject to the no-slip boundary condition if the domain of fluid is
finite, where $u=(u^{1},u^{2},u^{3})$ is the velocity vector field
of the fluid flow, $p(x,t)$ is the pressure at the instance $t$
and location $x$. J. Leray \cite{Leray-1934} demonstrated the existence
of a weak solution $u$ which belongs to the space $L^{\infty}\left([0,\infty),L^{2}(\mathbb{R}^{n})\right)$
and also to the space $L^{2}\left([0,\infty),H^{1}(\mathbb{R}^{n})\right)$.
The vorticity $\omega$ exists in $L^{2}$ space and formally, by
differentiating the Navier-Stokes equations, solves the vorticity
equation 
\begin{equation}
\frac{\partial}{\partial t}\omega+u\cdot\nabla\omega=\nu\Delta\omega+\omega\cdot\nabla u\label{eq:vor1}
\end{equation}
where the velocity $u$ and the vorticity $\omega$, which is too
a time dependent vector field $\omega=(\omega^{1},\omega^{2},\omega^{3})$,
are related by the definition that $\omega=\nabla\times u$. The resolution
of the Navier-Stokes equations remains to be an open mathematical
problem (see \cite{Lemarie-Rieusset-book-2016,von Wahl-book-1985}
for example), the current research has thus concentrated on the understanding
of the related partial differential equations and on developing numerical
approaches.

Observe that both the Navier-Stokes equations and the vorticity equations
may be put into the following form 
\begin{equation}
\left(\frac{\partial}{\partial t}-\nu\Delta+u\cdot\nabla\right)u=-\nabla p\label{eq:ns-11}
\end{equation}
and 
\begin{equation}
\left(\frac{\partial}{\partial t}-\nu\Delta+u\cdot\nabla\right)\omega=\omega\cdot\nabla u\label{eq:vot2}
\end{equation}
where the diffusion part is the same and is defined by the parabolic
operator 
\begin{equation}
L=\frac{\partial}{\partial t}-\nu\Delta+u\cdot\nabla.\label{eq:heat-01}
\end{equation}
The elliptic operator $\nu\Delta-u\cdot\nabla$ is the generator of
the so-called Taylor diffusion (see J. T. Taylor \cite{Taylor-1921,Taylor-1935})
of the flow of fluids. There are two non-linear terms appearing in
the Navier-Stokes equations and the vorticity equations, which determine
the turbulent nature of flows of fluids (see for example \cite{Monin-Yaglom-Vol1,Monin-Yaglom-Vol2}).
The parabolic operator $L$ has the capability of covering the so-called
non-linear convection mechanism \textendash{} the rate-of-strain (for
the Navier-Stokes equations \cite{von Wahl-book-1985,Libeberman-book--1996})
and the vorticity (in the case of the vorticity equations) can be
amplified even more rapidly by an increase of the velocity. It is
therefore important to study the parabolic equations associated with
the elliptic operator $A=\nu\Delta-u\cdot\nabla$, where $u$ is a
weak solution of the Navier-Stokes equations. The main feature here
is that $u(x,t)$ is a time-dependent vector field with little regularity,
which however is solenoidal, that is, for every $t$, $\nabla\cdot u(\cdot,t)=0$
in distribution sense, so that the formal adjoint $A^{*}=\nu\Delta+u\cdot\nabla$
is also a diffusion generator. These special features have significance,
and have been explored in several recent articles \cite{Hofmann-Mayboroda-2007,Osada-1987,Seregin-2012}
etc. for example. In this paper we give a through study for a class
of such parabolic equations.

Recall that in dimension three, a vector field $u=\left(u^{i}\right)$
is divergence-free, i.e. $\nabla\cdot u=0$, implies that its corresponding
two form (with respect to the Hodge star operation) $\star u$ is
closed, that is $d\star u=0$. In fact the divergence operator $\nabla\cdot$
coincides with the Hodge dual $\star d\star$ up to a sign, where
$d$ is the exterior differentiation. Therefore, according to the
Poincaré lemma, $\star u$ is exact, that is, there is a vector field
$b=\left(b^{i}\right)$ so that $\star u=db$. Hence $u$ coincides
with $\star db$ up to a sign. $\star b$ is a two form with components
$b^{ij}=\sum_{k}\varepsilon^{ijk}b^{k}$, where $\varepsilon^{ijk}$
is the usual Kronercker symbols of three elements. $\left(b^{ij}\right)$
is skew-symmetric, and 
\[
u\cdot\nabla=\sum_{i,j=1}^{3}\frac{\partial b^{ij}}{\partial x^{i}}\frac{\partial}{\partial x^{j}}.
\]
The elliptic operator $\nu\Delta-u\cdot\nabla$ can thus be put into
a divergence form 
\[
\sum_{i,j}\frac{\partial}{\partial x^{i}}\left(\nu\delta^{ij}-b^{ij}\right)\frac{\partial}{\partial x^{j}}\equiv\sum_{i,j}\frac{\partial}{\partial x^{i}}\left(A^{ij}\frac{\partial}{\partial x^{j}}\right)
\]
where $A=\left(A^{ij}\right)$ is not necessarily symmetric. The symmetric
part $\left(\nu\delta^{ij}\right)$ is uniformly elliptic, and the
skew-symmetric part $\left(b^{ji}\right)$ determines the divergence-free
drift vector field $u$.

In the present paper we develop a theory for the linear parabolic
equation 
\[
\sum_{i,j}\frac{\partial}{\partial x^{i}}\left(A^{ij}\frac{\partial u}{\partial x^{j}}\right)-\frac{\partial}{\partial t}u=0
\]
under very weak assumptions that $\left(A^{ij}\right)$ is uniformly
elliptic and its anti-symmetric part only belongs to the BMO space.

The paper is organized as following. In Section 2, we describe the
main result, that is the Aronson estimate which depends only on the
elliptic constant and the BMO norm of the anti-symmetric part of $A$,
which is the key tool of studying weak solutions. In Section 3, we
provide several results which will be used to prove the Aronson estimate
in our setting. These results are interesting by their own, including
several versions of the compensated compactness, and a density result
of the BMO space which seems new. In Section 4, we give the details
of the proof of the Aronson estimate, and in the final Section 5,
we study the weak solutions to the linear parabolic equations in divergence
form (but not necessarily symmetric) under weak assumptions. In particular,
we prove the existence and uniqueness of weak solutions to the parabolic
equation associated with a non-symmetric diffusion matrix $A=\left(A^{ij}\right)$.

\section{Aronson's estimate for non-symmetric parabolic equations}

Let us begin with the description of our framework. We consider the
following type of linear parabolic equations of second order 
\begin{equation}
\frac{\partial}{\partial t}u(x,t)-\left[A_{t}u\right](x,t)=0\quad\textrm{ on }\mathbb{R}^{n}\times[0,\infty)\label{eq:eq2-1}
\end{equation}
where 
\begin{equation}
A_{t}=\sum_{i,j=1}^{n}\frac{\partial}{\partial x^{i}}\left(A^{ij}(\cdot,t)\frac{\partial}{\partial x^{j}}\right)\label{eq:div-o1}
\end{equation}
is in divergence form but\emph{ not necessarily symmetric}, and their
associated diffusion processes in terms of fundamental solutions defined
by (\ref{eq:eq2-1}). There is a unique decomposition $A(x,t)=a(x,t)+b(x,t)$
such that $a(x,t)=\left(a^{ij}(x,t)\right)$ is symmetric, while $b(x,t)=\left(b^{ij}(x,t)\right)$
is skew-symmetric. We assume that $A$ is uniformly elliptic in the
following sense: there exists a constant $\lambda>0$ such that 
\begin{equation}
\lambda\vert\xi\vert^{2}\leq\sum_{i,j=1}^{n}a^{ij}(x,t)\xi_{i}\xi_{j}=\sum_{i,j=1}^{n}A^{ij}(x,t)\xi_{i}\xi_{j}\leq\frac{1}{\lambda}\vert\xi\vert^{2}\label{eq: uniform ellipticity}
\end{equation}
for any $\xi=\left(\xi_{i}\right)\in\mathbb{R}^{n}$, $x\in\mathbb{R}^{n}$
and $t\geq0$.

Let us first consider the regular case where $A^{ij}$ are smooth,
bounded and possess bounded derivatives of all orders on $\mathbb{R}^{n}\times[0,\infty)$.

Let $L=A_{t}-\frac{\partial}{\partial t}$ be the parabolic linear
operator associated with $\left(A^{ij}(x,t)\right)$. The formal adjoint
of $L$ is again a parabolic operator (with vanished zero order term)
given by 
\begin{equation}
L^{\star}=\sum_{i,j=1}^{n}\frac{\partial}{\partial x^{i}}\left(A^{ji}(\cdot,t)\frac{\partial}{\partial x^{j}}\right)+\frac{\partial}{\partial t}.\label{eq:f-ad0.1}
\end{equation}
where $A^{ji}=a^{ij}-b^{ij}$ is the transpose of $\left(A^{ij}\right)$.
It is known that (see A. Friedman \cite{Friedman-1964}, Theorem 11
and 12, Chapter 1), under the elliptic condition and smoothness assumptions
on $A^{ij}(x,t)$, there is a unique positive \emph{fundamental solution}
$\varGamma(x,t;\xi,\tau)$ of the parabolic operator $L$, and it
is smooth in $(x,t,\xi,\tau)$ on $0\leq\tau<t<\infty$ and $(x,\xi)\in\mathbb{R}^{n}\times\mathbb{R}^{n}$.
Recall that the following properties are satisfied.

1) $\varGamma(x,t;\xi,\tau)>0$ for any $0\leq\tau<t$ and $x,\xi\in\mathbb{R}^{n}$.

2) For every $\xi\in\mathbb{R}^{n}$ and $\tau\in[0,\infty)$, as
a function of $(x,t)\in\mathbb{R}^{n}\times(\tau,\infty)$, $u(x,t)\equiv\varGamma(x,t;\xi,\tau)$
solves the parabolic equation $Lu=0$ on $(\tau,\infty)\times\mathbb{R}^{n}$:
\begin{equation}
\sum_{i,j=1}^{n}\frac{\partial}{\partial x^{i}}\left(A^{ij}(x,t)\frac{\partial}{\partial x^{j}}\varGamma(x,t;\xi,\tau)\right)-\frac{\partial}{\partial t}\varGamma(x,t;\xi,\tau)=0\qquad\textrm{ on }\mathbb{R}^{n}\times(\tau,\infty).\label{eq:ga0.1}
\end{equation}

3) Chapman-Kolmogorov's equation holds 
\begin{equation}
\varGamma(x,t;\xi,\tau)=\int_{\mathbb{R}^{n}}\varGamma(x,t;z,s)\varGamma(z,s;\xi,\tau)dz.\label{eq:CK-01}
\end{equation}

4) For any bounded continuous function $f$ and $\tau\in[0,\infty)$,
it holds that 
\begin{equation}
\lim_{t\downarrow\tau}\int_{\mathbb{R}^{n}}f(\xi)\varGamma(x,t;\xi,\tau)d\xi=f(x)\label{eq:in-01}
\end{equation}
for every $x\in\mathbb{R}^{n}$.

For $0\leq\tau<t$, let $\varGamma_{\tau,t}$ denote the corresponding
linear operator defined by 
\begin{equation}
\varGamma_{\tau,t}f(x)=\int_{\mathbb{R}^{n}}f(\xi)\varGamma(x,t;\xi,\tau)d\xi\label{eq:ga02}
\end{equation}
where $f$ is Borel measurable, either non-negative, or/and bounded.
By (\ref{eq:CK-01}) 
\begin{equation}
\varGamma_{s,t}\circ\varGamma_{\tau,s}=\varGamma_{\tau,t}\label{eq:CK-02}
\end{equation}
for any $0\leq\tau<s<t$.

Define $\varGamma^{\star}(x,s;y,t)=\varGamma(y,t;x,s)$ for $t>s\geq0$.
Then $\varGamma^{\star}$ is the fundamental solution to $L^{\star}v=0$
in the sense that for every fixed $(y,t)$, as a function of $(x,s)$,
$\varGamma^{\star}$ solves the \emph{backward} parabolic equation
\begin{equation}
\sum_{i,j=1}^{n}\frac{\partial}{\partial x^{i}}\left(A^{ji}(x,s)\frac{\partial}{\partial x^{j}}\varGamma^{\star}(x,s;y,t)\right)+\frac{\partial}{\partial s}\varGamma^{\star}(x,s;y,t)=0\label{eq:adjoint-bas01}
\end{equation}
on $(x,s)\in\mathbb{R}^{n}\times[0,t)$. It follows that the fundamental
solution $\varGamma$ also solves the backward parabolic equation:
\begin{equation}
\sum_{i,j=1}^{n}\frac{\partial}{\partial\xi^{i}}\left(A^{ji}(\xi,\tau)\frac{\partial}{\partial\xi^{j}}\varGamma(x,t;\xi,\tau)\right)+\frac{\partial}{\partial\tau}\varGamma(x,t;\xi,\tau)=0\label{eq:bas02}
\end{equation}
which holds for any $\xi,x\in\mathbb{R}^{n}$ and $0<\tau<t$.

We are now in a position to state the key result of the present paper.

\begin{theorem}\label{-main-th1}There is a constant $M>0$ depending
only on the dimension $n$, the elliptic constant $\lambda>0$, and
the $L^{\infty}([0,\infty),\textrm{BMO})$ norm of the skew-symmetric
part $b^{ij}=\frac{1}{2}\left(A^{ij}-A^{ji}\right)$ such that 
\begin{equation}
\frac{1}{M(t-\tau)^{n/2}}\exp\left[-\frac{M|x-\xi|^{2}}{t-\tau}\right]\leq\varGamma(x,t;\xi,\tau)\leq\frac{M}{(t-\tau)^{n/2}}\exp\left[-\frac{|x-\xi|^{2}}{M(t-\tau)}\right]\label{eq:aron1}
\end{equation}
for any $0\leq\tau<t<\infty$ and $(x,\xi)\in\mathbb{R}^{n}\times\mathbb{R}^{n}$,
where the $L^{\infty}([0,\infty),\textrm{BMO})$ norm of $b^{ij}$
is defined by 
\[
\left\Vert b\right\Vert _{L^{\infty}\left(\textrm{BMO}\right)}=\sup_{t\geq0}\sqrt{\sum_{i<j}\left\Vert b^{ij}(\cdot,t)\right\Vert _{\textrm{BMO}}^{2}}.
\]
\end{theorem}

The fundamental heat kernel estimate (\ref{eq:aron1}) for parabolic
equations has a long history. Two side estimate (\ref{eq:aron1})
was first established in D. G. Aronson \cite{Aronson-1967,Aronson-1965}
for uniformly elliptic operators in divergence form where $A^{ij}$
is symmetric (so that $b^{ij}\equiv0$), his constant $M$ depends
only on the elliptic constant $\lambda$ and the dimension $n$. The
estimate (\ref{eq:aron1}) is therefore referred to as the Aronson
estimate. A weaker but global estimate similar to (\ref{eq:aron1})
under the same assumption as in D. G. Aronson \cite{Aronson-1967}
already appeared in the Appendix of J. Nash \cite{Nash-1958}. D.
G. Aronson \cite{Aronson-1967,Aronson-1968} indicated that his estimate
can be established for a general elliptic operator, and a written
proof is available in E. B. Fabes and D. W. Stroock \cite{Fabes-Stroock-1986},
D. W. Stroock \cite{Stroock-1988} and J. R. Norris and D. W. Stroock
\cite{Norris-Stroock-1991} too. In these papers, the Aronson estimate
(\ref{eq:aron1}) was established for the following type of uniformly
elliptic operator 
\[
\sum_{i,j=1}^{n}\frac{\partial}{\partial x^{i}}a^{ij}(x,t)\frac{\partial}{\partial x^{j}}+\sum_{i,j=1}^{n}a^{ij}(x,t)b_{j}(x,t)\frac{\partial}{\partial x^{i}}-\frac{\partial}{\partial x^{i}}\left(a^{ij}(x,t)\hat{b}_{j}(x,t)\right)+c(x,t)
\]
where $\left(a^{ij}\right)$ is symmetric and uniformly elliptic.
For this case, their constant $M$ depends on the dimension, the elliptic
constant $\lambda$ and the $L_{tx}^{\infty}$-norms of $b$, $\hat{b}$
and $c$.

A related topic to the Aronson estimate is the regularity of solutions
to the parabolic equation $Lu=0$ (see for a complete survey of classical
results \cite{Ladyzenskaja-book-1968}). If the elliptic operator
is symmetric and is in divergence form, it was J. Nash \cite{Nash-1958}
who proved the Hölder continuity of bounded solutions and also proved
that the Hölder exponent depends only on the dimension and the elliptic
constant $\lambda$. Under the same setting as that of J. Nash \cite{Nash-1958},
in 1964, J. Moser \cite{Moser-1964a} established the Harnack inequality
for positive solutions of the parabolic equation $Lu=0$, based on
which G. Aronson was able to derive his estimate (\ref{eq:aron1}).
E. B. Fabes and D. W. Stroock \cite{Fabes-Stroock-1986} showed that
J. Moser's Harnack inequality can be derived from Aronson estimate
together with J. Nash's idea, and D. W. Stroock \cite{Stroock-1988}
further demonstrated that both the Hölder continuity of classical
solutions and Moser's Harnack inequality for positive solutions can
be established by utilizing the two side Aronson estimate (\ref{eq:aron1}).
J. Nash's idea in \cite{Nash-1958} and the techniques in J. Moser
\cite{Moser-1964a,Moser-1964b,Moser-1971}\emph{ }have been investigated
intensively during the past decades. Many excellent results have been
obtained in more general settings, but mainly under the symmetric
setting of Dirichlet forms \cite{Fukushima-1980}. See for example
A. A. Grigor'yan \cite{Grigoryan-1991}, E. B. Davies \cite{Davies-book-1989}
and D. W. Stroock \cite{Stroock-book-2008} for a small sample of
references, and see also the literature therein.

The case that $\left(A^{ij}\right)$ is non-symmetric has received
intensive study only recently, due to the connection with the Navier-Stokes
equations and the blow-up behavior of their solutions. In H. Osada
\cite{Osada-1987}, the Aronson estimate (\ref{eq:aron1}) was obtained
for an elliptic operator in divergence form as ours, where $\left(A^{ij}\right)$
may not be symmetric, his constant $M$ in (\ref{eq:aron1}) however
depends on the dimension $n$, the elliptic constant $\lambda$ and
the $L_{tx}^{\infty}$-norm of the skew-symmetric part $(b^{ij})$.
In a recent work by G. Seregin, L. Silvestre, V. Šverák, and A. Zlatoš
\cite{Seregin-2012}, who noticed that a large portion of Nash's arguments
also work for an elliptic operator with divergence-free drifts, i.e.
where the elliptic operator has a form $\Delta+u(x,t)\cdot\nabla$
such that $\nabla\cdot u(\cdot,t)=0$. In particular, they mentioned
that the fundamental solution $\varGamma$ of the heat operator 
\[
\Delta+u(x,t)\cdot\nabla-\frac{\partial}{\partial t}
\]
satisfies the diagonal decay estimate 
\[
\varGamma(x,t;x,\tau)\leq\frac{C}{(t-\tau)^{\frac{n}{2}}}
\]
for all $t>\tau\geq0$. They further proved the Harnack inequality
in this case, and their constants depend on the dimension and the
$L_{t}^{\infty}\text{BMO}_{x}^{-1}$ of the vector field.

Our work was motivated by the observation made by G. Seregin and etc.
\cite{Seregin-2012}, and the approach put forward by E. Davies \cite{Davies-1987},
E. B. Fabes and D. Stroock \cite{Fabes-Stroock-1986}, and D. W. Stroock
\cite{Stroock-1988}. We follow the approach in E. B. Davies and D.
W. Stroock to the non-symmetric case, and adopt their arguments to
our case by overcoming the difficulties arising from the singularities
of the skew-symmetric part $(b^{ij})$. In a sense, the present work
is to complete the program initiated by G. Seregin and etc. \cite{Seregin-2012}
by bringing in the techniques developed over years by various authors.

As in D. W. Stroock \cite{Stroock-1988}, as consequences of the Aronson
estimate, we have the following continuity theorem and the Harnack
inequality.

\begin{theorem}\label{cor: continuity of transition probability}
There exist $C>0$ and $\alpha\in(0,1)$ depending only on the dimension
$n$, the elliptic constant $\lambda$ and the $L_{t}^{\infty}\textrm{BMO}_{x}$-norm
of the skew-symmetric part $(b^{ij})$, such that for every $\delta>0$
\begin{equation}
\vert\varGamma(x,t;\xi,\tau)-\varGamma(x',t';\xi',\tau)\vert\leq\frac{C}{\delta^{n}}(\vert t'-t\vert\vee\vert x'-x\vert\vee\vert\xi'-\xi\vert)^{\alpha}
\end{equation}
for all $\tau\geq0$, $(t',x',\xi'),(t,x,\xi)\in[s+\delta^{2})\times\mathbb{R}^{n}\times\mathbb{R}^{n}$
with $\vert t'-t\vert\vee\vert x'-x\vert\vee\vert y'-y\vert\leq\delta$.
\end{theorem}

\begin{theorem}{[}Harnack Inequality{]} There exists a constant $C>0$
depending only on $n,\lambda$ and $\Vert b\Vert_{L^{\infty}(BMO)}$
such that given any $v\in L^{2}(\mathbb{R}^{n})$ with $v\geq0$ and
set $u(t,x)=\varGamma_{\tau,t}v(x)$, we have 
\begin{equation}
\sup_{[s,s+R^{2}]\times B(x_{0},R)}u(t,x)\leq C\inf_{[s+3R^{2},s+4R^{2}]\times B(x_{0},R)}u(t,y)\label{eq: Harnack inequality}
\end{equation}
for any $R>0$, $(x,s)\in\mathbb{R}^{n}\times[\tau,\infty)$. \end{theorem}

The Harnack inequality is also established by G. Seregin and etc.
in \cite{Seregin-2012} under a bit additional technical conditions
than those stated in the theorem above.

\section{Several technique facts}

In this and next several sections, we are going to prove the main
result, Aronson estimate. In this section, we prove several technique
facts which will be needed in the proof of the main result.

The first result we need is a variation of Coifman-Meyer's compensated
compactness Theorem \cite{Coifman-1971,Coifman-1992,Lemarie-Rieusset-book-2016}
which highlights the importance of the Hardy spaces in the study of
partial differential equations.

We first recall some facts on $\textrm{BMO}$ functions \cite{Stein-book-1970,John-Nirenberg-1961}.
A function $f$ is in $\textrm{BMO}(\mathbb{R}^{n})$ if 
\begin{equation}
\Vert f\Vert_{\textrm{BMO}}=\sup_{B\subset\mathbb{R}^{n}}\frac{1}{\vert B\vert}\int_{B}\vert f(x)-f_{B}\vert\;dx<\infty
\end{equation}
where $f_{B}=\frac{1}{\vert B\vert}\int_{B}f(x)\;dx$ and the supremum
is taken over all open balls $B\in\mathbb{R}^{n}$ (in what follows,
$B_{r}(x)$ or $B(x,r)$ denotes the ball centered at $x$ with radius
$r$). If define another norm 
\begin{equation}
\Vert f\Vert_{\textrm{BMO}_{p}}^{p}=\sup_{B\subset\mathbb{R}^{n}}\frac{1}{\vert B\vert}\int_{B}\vert f(x)-f_{B}\vert^{p}\;dx<\infty
\end{equation}
for any $1\leq p<\infty$, John-Nirenberg inequality \cite{John-Nirenberg-1961}
(see also for example, Appexdix in D. W. Stroock and S. R. S. Varadhan
\cite{Stroock-Varadhan-book-1979}) implies that $\Vert\cdot\Vert_{\textrm{BMO}_{p}}$
are equivalent for different $p$.

The original version of the compensated compactness Theorem, which
will be used in our proof of the lower bound of Aronson estimate,
can be stated as following

\begin{proposition}\label{the: compensated compactness} Let vector
fields $E,B$ satisfy $E\in L^{p}(\mathbb{R}^{n})^{n}$, $B\in L^{q}(\mathbb{R}^{n})^{n}$
with $\frac{1}{p}+\frac{1}{q}=1$ ($p\geq1$, $q\geq1$) and $\nabla\cdot E=0$,
$\nabla\times B=0$. Then $E\cdot B\in\mathcal{H}^{1}$ where $\mathcal{H}^{1}$
is the Hardy space, and 
\begin{equation}
\Vert E\cdot B\Vert_{\mathcal{H}^{1}}\leq C\Vert E\Vert_{p}\Vert B\Vert_{q}.
\end{equation}
In particular, there is a constant $C$ depending on the dimension
$n$ and $p>1$ such that 
\begin{equation}
\left\Vert \nabla f\times\nabla g\right\Vert _{\mathcal{H}^{1}}\leq C\left\Vert \nabla f\right\Vert _{p}\left\Vert \nabla g\right\Vert _{q}\label{eq:dr1}
\end{equation}
for any $f,g\in C_{0}^{\infty}(\mathbb{R}^{n})$, where $\frac{1}{p}+\frac{1}{q}=1$.
Hence 
\begin{equation}
\left|\int_{\mathbb{R}^{n}}\left\langle \nabla f(x),b(x)\cdot\nabla g(x)\right\rangle dx\right|\leq C\left\Vert b\right\Vert _{\textrm{BMO}}\left\Vert \nabla f\right\Vert _{2}\left\Vert \nabla g\right\Vert _{2}\label{eq:c-c-e}
\end{equation}
for any $f,g\in H^{1}\left(\mathbb{R}^{n}\right)$ and for any $b=\left(b^{ij}\right)\in\textrm{BMO}$
which is skew-symmetric, $b^{ij}=-b^{ji}$. \end{proposition}

To prove the upper bound of Aronson estimate, we need the following
estimate, in the same spirit of compensated compactness.

\begin{proposition}\label{the: compensated compactness 2} There
is a universal constant $C$ depending only on the dimension $n$,
such that 
\begin{equation}
\Vert f\nabla_{\xi}f\Vert_{\mathcal{H}^{1}}\leq C\vert\xi\vert\Vert\nabla f\Vert_{2}\Vert f\Vert_{2}.
\end{equation}
for any $f\in H^{1}(\mathbb{R}^{n})=W^{1,2}(\mathbb{R}^{n})$ and
$\xi\in\mathbb{R}^{n}$, where $\left\Vert \cdot\right\Vert _{\mathcal{H}^{1}}$
denotes the Hardy norm.\end{proposition}

\begin{proof} Let $h$ be any smooth non-negative function on $\mathbb{R}^{n}$,
with its support in the unit ball $B_{1}(0)$ such that $\int_{\mathbb{R}^{n}}h(x)dx=1$,
and for $t>0$, $h_{t}(x)=t^{-n}h\left(x/t\right)$. Notice that $f\nabla_{\xi}f=\frac{1}{2}\nabla\cdot(f^{2}\xi)$
in $L^{1}(\mathbb{R}^{n})$, so 
\begin{eqnarray*}
h_{t}\ast(f\nabla_{\xi}f)(x) & = & \frac{1}{2}\int_{B_{t}(x)}\nabla h_{t}(x-y)\cdot\xi f^{2}(y)dy\\
 & = & \int_{B_{t}(x)}\frac{1}{t^{n+1}}\nabla h\left(\frac{x-y}{t}\right)\cdot\xi f^{2}(y)dy\\
 & = & \int_{B_{t}(x)}\frac{1}{t^{n+1}}\nabla h\left(\frac{x-y}{t}\right)\cdot\xi f(y)\left[f(y)-\fint_{B_{t}(x)}f\right]dy\\
 &  & +\int_{B_{t}(x)}\frac{1}{t^{n+1}}\nabla h\left(\frac{x-y}{t}\right)\cdot\xi f(y)\fint_{B_{t}(x)}f(y)dy\\
 & = & I_{1}+I_{2}
\end{eqnarray*}
where $\fint_{B_{t}(x)}$ denotes the average integral over the ball
$B_{t}(x)$, that is, $|B_{t}(x)|^{-1}\int_{B_{t}(x)}$. For the first
term on the right-hand side, we have 
\begin{equation}
|I_{1}|\leq C\left[\fint_{B_{t}(x)}\vert\xi f\vert^{\alpha}\right]^{\frac{1}{\alpha}}\left[\fint_{B_{t}(x)}\left|\left(f(y)-\fint_{B_{t}(x)}f\right)t^{-1}\right|^{\alpha'}\right]^{\frac{1}{\alpha'}}
\end{equation}
where $\alpha\in[1,2)$, $\frac{1}{\alpha}+\frac{1}{\alpha'}=1$.
Choose $\alpha,\beta$ such that $1\leq\alpha,\beta<2$ and $\frac{1}{\alpha}+\frac{1}{\beta}=1+\frac{1}{n}$.
Then by the Sobolev-Poincaré inequality, we have 
\begin{equation}
\left[\fint_{B_{t}(x)}\left|\left(f-\fint_{B_{t}(x)}f\right)t^{-1}\right|^{\alpha'}\right]^{\frac{1}{\alpha'}}\leq C\left(\fint_{B_{t}(x)}\vert\nabla f\vert^{\beta}\right)^{\frac{1}{\beta}}.
\end{equation}
For the second term on the right-hand side, we integrate by part again
to obtain 
\begin{equation}
|I_{2}|=\left|\int_{B_{t}(x)}h\left(\frac{x-y}{t}\right)\frac{1}{t^{n}}\cdot\textrm{div}(\xi f(y))\fint_{B_{t}(x)}f(y)dy\right|\leq C\vert\xi\vert\fint_{B_{t}(x)}\vert\nabla f(y)\vert dy\fint_{B_{t}(x)}f(y)dy.
\end{equation}
By using these estimates we thus conclude that 
\begin{eqnarray*}
\sup_{t>0}\vert\{h_{t}\ast(\xi f\cdot\nabla f)\}(x)\vert & \leq & C\vert\xi\vert\sup_{t>0}\left(\fint_{B_{t}(x)}\vert f\vert^{\alpha}\right)^{\frac{1}{\alpha}}\sup_{t>0}\left(\fint_{B_{t}(x)}\vert\nabla f\vert^{\beta}\right)^{\frac{1}{\beta}}\\
 &  & +C\vert\xi\vert\sup_{t>0}\left(\fint_{B_{t}(x)}\vert f\vert\right)\sup_{t>0}\left(\fint_{B_{t}(x)}\vert\nabla f\vert\right)\\
 & = & C\vert\xi\vert[M(\vert f\vert^{\alpha})^{\frac{1}{\alpha}}M((\vert\nabla f\vert^{\beta})^{\frac{1}{\beta}}+M(\vert f\vert)M((\vert\nabla f\vert)]
\end{eqnarray*}
where $M(f)$ is the maximal function. Since $1\leq\alpha<2,\;1<\beta<2$,
we have $\Vert M(\vert f\vert^{\alpha})^{\frac{1}{\alpha}}\Vert_{2}\leq C\Vert f\Vert_{2}$,
$\Vert M(\vert\nabla f\vert^{\beta})^{\frac{1}{\beta}}\Vert_{2}\leq C\Vert\nabla f\Vert_{2}$
and similarly $\Vert M(\vert f\vert)\Vert_{2}\leq C\Vert f\Vert_{2}$,
$\Vert M(\vert\nabla f\vert)\Vert_{2}\leq C\Vert\nabla f\Vert_{2}$.
So $\sup_{t>0}\vert h_{t}\ast(\xi f\cdot\nabla f)\vert\in L^{1}$
and 
\begin{equation}
\Vert f\cdot\nabla_{\xi}f\Vert_{\mathcal{H}^{1}}\leq C\vert\xi\vert\Vert\nabla f\Vert_{2}\Vert f\Vert_{2}.
\end{equation}
\end{proof}

Given a function $b\in L^{\infty}([0,\infty),\textrm{BMO}(\mathbb{R}^{n}))$,
we want to approximate it by a mollified sequence, which is not trivial
as it looks. A simple example is a vector field $b(t)$ which depends
only on $t$, not on the space variables. Then it may not be in $L_{loc}^{1}$
and there is no approximations by mollifying sequences. However, the
problem considered here allow us to add a constant to it, i.e. consider
$b(t,x)+f(t)$, where $f(t)$ is skew-symmetric so that it will not
alter the weak solution formulation of the corresponding parabolic
equations. So by subtracting the mean value of $b$ on a unit ball,
we may assume that 
\begin{equation}
b_{B(0,1)}(t)=\fint_{B(0,1)}b(t,x)\;dx=0\qquad\mbox{for all }t\in[0,\infty).\label{eq: a nice version}
\end{equation}
Then for any $r>0$
\begin{eqnarray}
\vert b_{B(r)}(t)\vert & = & \vert b_{B(r)}(t)-b_{B(1)}(t)\vert=\vert\fint_{B(1)}b_{B(r)}(t)-b(t,x)\;dx\vert\\
 & \leq & r^{n}\fint_{B(r)}\vert b_{B(r)}(t)-b(t,x)\vert\;dx\leq r^{n}\Vert b\Vert_{L^{\infty}(\textrm{BMO}(\mathbb{R}^{n}))}
\end{eqnarray}
where $B(r)=B(0,r)$. By the definition of $\textrm{BMO}$ functions,
we have 
\begin{eqnarray}
\fint_{B(r)}\vert b(t,x)-b_{B(r)}(t)\vert^{p}\;dx\leq C\Vert b\Vert_{L^{\infty}(\textrm{BMO}(\mathbb{R}^{n}))}^{p},
\end{eqnarray}
which implies that $b\in L_{loc}^{p}([0,\infty)\times\mathbb{R}^{n})$
for any $1\leq p<\infty$.

\begin{proposition}\label{prop-bmo} Take $\Phi\in C_{0}^{\infty}(B(1))$
and $\eta\in C_{0}^{\infty}([-1,1])$ with $\Phi,\eta\geq0$ and 
\[
\int_{B(1)}\Phi(x)\;dx=\int_{[-1,1]}\eta(t)\;dt=1.
\]
Let $\Phi_{\epsilon}(x)=\frac{1}{\epsilon^{n}}\Phi(\frac{x}{\epsilon})$
and $\eta_{\epsilon}(t)=\frac{1}{\epsilon}\eta(\frac{t}{\epsilon})$.
Suppose $b\in L^{\infty}(\textrm{BMO}(\mathbb{R}^{n}))$ and satisfies
\eqref{eq: a nice version}. Define 
\begin{equation}
b_{\epsilon}(t,x)=\int_{-\epsilon}^{+\epsilon}\int_{B(0,\epsilon)}\Phi_{\epsilon}(y)\eta_{\epsilon}(s)b(t-s,x-y)\;dyds.
\end{equation}
Then $b_{\epsilon}\rightarrow b$ locally in $L^{p}$ for any $1\leq p<\infty$,
and 
\begin{equation}
\Vert b_{\epsilon}\Vert_{L^{\infty}(\textrm{BMO}(\mathbb{R}^{n}))}\leq\Vert b\Vert_{L^{\infty}(\textrm{BMO}(\mathbb{R}^{n}))}.\label{eq:non-de-01}
\end{equation}
\end{proposition}

\begin{proof} Let $x_{0}$, and $r>0$ be fixed. Let $\fint$ denote
the average integral over $B(x_{0},r)$, that is, 
\[
\fint\phi(y)dy=\fint_{B(x_{0},r)}\phi(y)\;dy.
\]
For $x\in B(x_{0},r)$ we have 
\begin{eqnarray*}
 &  & \left|b_{\epsilon}(t,x)-\fint b_{\epsilon}(t,y)\;dy\right|\\
 & = & \left|\int_{-\epsilon}^{+\epsilon}\int_{B(0,\epsilon)}\Phi_{\epsilon}(y)\eta_{\epsilon}(s)b(t-s,x-y)\;dyds\right.\\
 &  & \left.-\fint_{B(x_{0},r)}\int_{-\epsilon}^{+\epsilon}\int_{B(0,\epsilon)}\Phi_{\epsilon}(z)\eta_{\epsilon}(s)b(t-s,y-z)\;dzdsdy\right|\\
 & = & \left|\int_{-\epsilon}^{+\epsilon}\int_{B(0,\epsilon)}\Phi_{\epsilon}(y)\eta_{\epsilon}(s)[b(t-s,x-y)-\fint b(t-s,z-y)\;dz]dyds\right|\\
 & \leq & \int_{-\epsilon}^{+\epsilon}\int_{B(0,\epsilon)}\Phi_{\epsilon}(y)\eta_{\epsilon}(s)\left|b(t-s,x-y)-\fint b(t-s,z-y)\;dz\right|dyds,
\end{eqnarray*}
so that 
\begin{eqnarray*}
 &  & \fint\left|b_{\epsilon}(t,x)-\fint b_{\epsilon}(t,y)\;dy\right|\;dx\\
 & \leq & \fint\int_{-\epsilon}^{+\epsilon}\int_{B(0,\epsilon)}\Phi_{\epsilon}(y)\eta_{\epsilon}(s)\left|b(t-s,x-y)-\fint b(t-s,z-y)\;dz\right|dydsdx\\
 & = & \int_{-\epsilon}^{+\epsilon}\int_{B(0,\epsilon)}\Phi_{\epsilon}(y)\eta_{\epsilon}(s)\fint\left|b(t-s,x-y)-\fint b(t-s,z-y)\;dz\right|\;dxdyds\\
 & \leq & \int_{-\epsilon}^{+\epsilon}\int_{B(0,\epsilon)}\Phi_{\epsilon}(y)\eta_{\epsilon}(s)\Vert b\Vert_{L^{\infty}(\textrm{BMO}(\mathbb{R}^{n}))}\;dyds\\
 & = & \Vert b\Vert_{L^{\infty}(\textrm{BMO}(\mathbb{R}^{n}))}.
\end{eqnarray*}
Hence we have proved $\Vert b_{\epsilon}\Vert_{L^{\infty}(\textrm{BMO}(\mathbb{R}^{n}))}\leq\Vert b\Vert_{L^{\infty}(\textrm{BMO}(\mathbb{R}^{n}))}$.
\end{proof}

The lattice property in the proposition below of the BMO space should
be well known, for completeness, a proof is attached.

\begin{proposition} Suppose $f,g\in\textrm{BMO}(\mathbb{R}^{n})$,
then $f\wedge g$ and $f\vee g\in\textrm{BMO}(\mathbb{R}^{n})$. Moreover,
we have 
\begin{equation}
\Vert f\wedge g\Vert_{BMO}\leq C\left(\Vert f\Vert_{\textrm{BMO}}+\Vert g\Vert_{\textrm{BMO}}\right)
\end{equation}
where $C$ only depends on $n$. \end{proposition}

\begin{proof} Here we only prove it for $f\wedge g$ and $f\vee g$
follows similar proof. Observe that for any $a,b,c,d\in\mathbb{R}$,
we have 
\begin{equation}
\vert a\wedge b-c\wedge d\vert\leq\vert a-c\vert+\vert b-d\vert.
\end{equation}
Hence for any ball $B$, 
\begin{eqnarray*}
\frac{1}{\vert B\vert}\int_{B}\vert f\wedge g(x)-(f\wedge g)_{B}\vert^{2}\;dx & \leq & \frac{1}{\vert B\vert}\int_{B}\vert f\wedge g(x)-f_{B}\wedge g_{B}\vert^{2}\;dx\\
 & \leq & \frac{2}{\vert B\vert}\int_{B}\vert f(x)-f_{B}\vert^{2}\;dx+\frac{2}{\vert B\vert}\int_{B}\vert g(x)-g_{B}\vert^{2}\;dx\\
 & \leq & C(\Vert f\Vert_{\textrm{BMO}}^{2}+\Vert g\Vert_{\textrm{BMO}}^{2})
\end{eqnarray*}
and the proof is done. \end{proof}

\section{Proof of Aronson's estimate}

The proof follows the main lines as in D. W. Stroock \cite{Stroock-1988}
and in particular E. B. Davies \cite{Davies-1987} from which a clever
use of the $h$-transforms from harmonic analysis is borrowed, while
we need to overcome several difficulties since $A$ is non-symmetric
and the skew-symmetry part $b$ is singular. These ideas are mainly
due to J. Nash \cite{Nash-1958}, J. Moser \cite{Moser-1960,Moser-1964a,Moser-1964b,Moser-1971}.

Let us begin with the proof of the upper bound.

\subsection{Proof of the upper bound}

In this part we show the upper bound:

\begin{equation}
\varGamma(x,t;\xi,\tau)\leq\frac{C}{(t-\tau)^{\frac{n}{2}}}\exp\left[-\frac{|x-\xi|^{2}}{C(t-\tau)}\right]\label{eq:aron1-1}
\end{equation}
for any $t>\tau$ and $x,\xi\in\mathbb{R}^{n}$, where $C$ depends
only on $n$, $\lambda$ and $\left\Vert b\right\Vert _{L^{\infty}\left(\textrm{BMO}\right)}$.

The main idea of E. B. Davies \cite{Davies-1987} is to consider the
$h$-transform of the fundamental solution $\varGamma\left(x,t;\xi,\tau\right)$
and apply Nash and Moser's iteration to the $h$-transforms of the
fundamental solution $\varGamma$. J. Nash's idea is to iterate the
$L^{p}$-norms of solutions to parabolic equations, and to control
the growth of the $L^{p}$-norms. The main ingredient in J. Nash's
argument is the clever use of the Nash inequality 
\begin{equation}
\Vert u\Vert_{2}^{2+\frac{4}{n}}\leq C_{n}\Vert\nabla u\Vert_{2}^{2}\Vert u\Vert_{1}^{\frac{4}{n}},\qquad\forall u\in L^{1}(\mathbb{R}^{n})\cap H^{1}(\mathbb{R}^{n})\label{eq:nash1}
\end{equation}
where $C_{n}>0$ is a constant depending only on the dimension $n$.

The Nash iteration is neatly described as the following (D. Stroock
\cite{Stroock-1988}, Lemma I.1.14, page 322).

\begin{lemma}\label{lem: inequality 1(2)} Given positive numbers
$c_{1}$, $c_{2}$, $\beta$ and $p\geq2$. Let $w$ be positive and
non-decreasing, continuous on $[0,\infty)$, and $u$ be positive
and has continuous derivatives on $[0,\infty)$. Suppose the following
differential inequality holds: 
\begin{equation}
u'(t)\leq-\frac{c_{1}}{p}\frac{t^{(p-2)}u(t)^{1+\beta p}}{w(t)^{\beta p}}+c_{2}pu(t),\quad t\geq0.
\end{equation}
Then there exists a $K(c_{1},\beta)>0$ such that 
\begin{equation}
t^{(p-1)/\beta p}u(t)\leq\left(\frac{Kp^{2}}{\delta}\right)^{\frac{1}{\beta p}}e^{\frac{c_{2}\delta t}{p}}w(t),\quad t\geq0
\end{equation}
for every $\delta\in(0,1]$. \end{lemma}

The above iteration procedure works in a very general setting, and
has been explored since the publication of J. Nash' paper \cite{Nash-1958},
and it is still the major ingredient in our proof. It is surprising
that they work well even in our setting where the diffusion is very
singular.

Fortunately as well, E. B. Davies' idea \cite{Davies-1987,Davies-book-1989}
also works well for our parabolic equations. Following E. B. Davies
\cite{Davies-1987} and D. Stroock \cite{Stroock-1988}, given a smooth
function $\psi$ on $\mathbb{R}^{n}$, consider 
\begin{equation}
\varGamma^{\psi}\left(x,t;\xi,\tau\right)=e^{-\psi(x)}\varGamma\left(x,t;\xi,\tau\right)e^{\psi(\xi)}\label{eq:dav01}
\end{equation}
and the linear operator 
\[
\varGamma_{\tau,t}^{\psi}f(x)=\int_{\mathbb{R}^{n}}f(\xi)e^{-\psi(x)}\varGamma\left(x,t;\xi,\tau\right)e^{\psi(\xi)}d\xi
\]
which is defined for non-negative Borel measurable $f$ , and for
$f$ which is smooth with a compact support. It is easy to see that
the adjoint operator of $\varGamma_{\tau,t}^{\psi}$ can be identified
as the following integral operator 
\[
\varGamma_{\tau,t}^{\psi\dagger}f(x)=\int_{\mathbb{R}^{n}}f(\xi)e^{\psi(x)}\varGamma\left(\xi,t;x,\tau\right)e^{-\psi(\xi)}d\xi.
\]
That is 
\[
\left\langle \varGamma_{\tau,t}^{\psi}f,g\right\rangle =\left\langle f,\varGamma_{\tau,t}^{\psi\dagger}g\right\rangle 
\]
for any smooth functions $f$ and $g$ with compact supports.

\begin{lemma} Let $T>0,\tau\geq0$. Let $f\in C_{0}^{\infty}(\mathbb{R}^{n})$
be non-negative, and $\psi(x)=\alpha\cdot x$ where $\alpha\in\mathbb{R}^{n}$.
Define 
\[
f_{t}(x)=\varint_{\mathbb{R}^{n}}f(\xi)e^{\psi(x)-\psi(\xi)}\varGamma(x,t;\xi,\tau)d\xi=\varGamma_{\tau,t}^{\psi}f(x)
\]
for $t>\tau$, and 
\[
f_{t}^{\dagger}(x)=\varint_{\mathbb{R}^{n}}f(\xi)e^{\psi(x)-\psi(\xi)}\varGamma(\xi,T;x,T-t)d\xi
\]
for $t\in(0,T]$.

There is a constant $C>0$ depending only on $n$, such that for any
$p\geq1$, 
\begin{equation}
\frac{d}{dt}\left\Vert f_{t}\right\Vert _{2p}^{2p}\leq-\lambda\left\Vert \nabla f_{t}^{p}\right\Vert _{2}^{2}+C_{B}\frac{p^{2}\left|\alpha\right|^{2}}{\lambda}\left\Vert f_{t}\right\Vert _{2p}^{2p}\label{eq: energy inequality 1(2)}
\end{equation}
$t>\tau$, and 
\begin{equation}
\frac{d}{dt}\left\Vert f_{t}^{\dagger}\right\Vert _{2p}^{2p}\leq-\lambda\left\Vert \nabla f_{t}^{\dagger p}\right\Vert _{2}^{2}+C_{B}\frac{p^{2}\left|\alpha\right|^{2}}{\lambda}\left\Vert f_{t}^{\dagger}\right\Vert _{2p}^{2p}\label{eq: energy inequality 1(2)-1}
\end{equation}
for all $t\in(0,T]$, where 
\[
C_{B}=2C\left\Vert b\right\Vert _{L^{\infty}\left(\textrm{BMO}\right)}^{2}+2
\]
which depends only on $n$ and the $L_{t}^{\infty}\textrm{BMO}_{x}$
of $b(x,t)$. \end{lemma}

\begin{proof}We may assume that $\tau=0$ without lose of generality,
so that 
\begin{align*}
\left\Vert f_{t}\right\Vert _{2p}^{2p} & =\int_{\mathbb{R}^{n}}\left(\varint_{\mathbb{R}^{n}}f(\xi)\varGamma^{\psi}(x,t;\xi,0)d\xi\right)^{2p}dx\\
 & =\int_{\mathbb{R}^{n}}\left(\varint_{\mathbb{R}^{n}}f(\xi)e^{-\psi(x)+\psi(\xi)}\varGamma(x,t;\xi,0)d\xi\right)^{2p}dx.
\end{align*}
Differentiating $\left\Vert f_{t}\right\Vert _{2p}^{2p}$ to obtain
\[
\frac{d}{dt}\left\Vert f_{t}\right\Vert _{2p}^{2p}=2p\int_{\mathbb{R}^{n}}f_{t}(x)^{2p-1}\left(\varint_{\mathbb{R}^{n}}f(\xi)e^{-\psi(x)+\psi(\xi)}\frac{\partial}{\partial t}\varGamma(x,t;\xi,0)d\xi\right)dx
\]
and by using the equation (\ref{eq:ga0.1}) we have 
\[
\frac{d}{dt}\left\Vert f_{t}\right\Vert _{2p}^{2p}=2p\int_{\mathbb{R}^{n}}\left[f_{t}(x)^{2p-1}\varint_{\mathbb{R}^{n}}f(\xi)e^{-\psi(x)+\psi(\xi)}\nabla_{x}\cdot\left(A(x,t)\nabla_{x}\varGamma(x,t;\xi,0)\right)d\xi\right]dx.
\]
Similarly, by using the backward equation (\ref{eq:bas02}) we have
\[
\frac{d}{dt}\left\Vert f_{t}^{\dagger}\right\Vert _{2p}^{2p}=2p\int_{\mathbb{R}^{n}}\left[f_{t}^{\dagger}(x)^{2p-1}\varint_{\mathbb{R}^{n}}f(\xi)e^{\psi(x)-\psi(\xi)}\nabla_{x}\cdot\left(A^{T}(x,t)\nabla_{x}\varGamma(\xi,T;x,T-t)\right)d\xi\right]dx.
\]
By using the Fubini theorem, then performing integration by parts
we therefore have 
\begin{align*}
\frac{1}{2p}\frac{d}{dt}\left\Vert f_{t}\right\Vert _{2p}^{2p} & =\int_{\mathbb{R}^{n}}\left(e^{\psi(\xi)}f(\xi)\varint_{\mathbb{R}^{n}}e^{-\psi(x)}f_{t}(x)^{2p-1}\nabla_{x}\cdot\left(A(x,t)\nabla_{x}\varGamma(x,t;\xi,0)\right)dx\right)d\xi\\
 & =\int_{\mathbb{R}^{n}}f_{t}(x)^{2p}\left\langle \nabla\psi,a(x,t)\cdot\nabla\psi\right\rangle dx\\
 & -\frac{2p-1}{p^{2}}\int_{\mathbb{R}^{n}}\left\langle \nabla f_{t}(x)^{p},a(x,t)\cdot\nabla f_{t}(x)^{p}\right\rangle dx\\
 & -\frac{2(p-1)}{p}\int_{\mathbb{R}^{n}}f_{t}(x)^{p}\left\langle \nabla f_{t}(x)^{p},a(x,t)\cdot\nabla\psi\right\rangle dx\\
 & -2\int_{\mathbb{R}^{n}}f_{t}(x)^{p}\left\langle \nabla f_{t}(x)^{p},b(x,t)\cdot\nabla\psi\right\rangle dx\\
 & =I_{1}-I_{2}-I_{3}-I_{4}.
\end{align*}
and similarly we have 
\begin{align*}
\frac{1}{2p}\frac{d}{dt}\left\Vert f_{t}^{\dagger}\right\Vert _{2p}^{2p} & =\int_{\mathbb{R}^{n}}f_{t}^{\dagger}(x)^{2p}\left\langle \nabla\psi,a(x,t)\cdot\nabla\psi\right\rangle dx\\
 & -\frac{2(2p-1)}{p}\int_{\mathbb{R}^{n}}\left\langle \nabla f_{t}^{\dagger}(x)^{p},a(x,t)\cdot\nabla f_{t}^{\dagger}(x)^{p}\right\rangle dx\\
 & +\frac{2(p-1)}{p}\int_{\mathbb{R}^{n}}f_{t}^{\dagger}(x)^{p}\left\langle \nabla\psi,a(x,t)\cdot\nabla f_{t}^{\dagger}(x)^{p}\right\rangle dx\\
 & -2\int_{\mathbb{R}^{n}}f_{t}^{\dagger}(x)^{p}\left\langle \nabla\psi,b(x,t)\cdot\nabla f_{t}^{\dagger}(x)^{p}\right\rangle dx.
\end{align*}
Since $\frac{d}{dt}\left\Vert f_{t}\right\Vert _{2p}^{2p}$ and $\frac{d}{dt}\left\Vert f_{t}^{\dagger}\right\Vert _{2p}^{2p}$
are similar, we only need to prove (\ref{eq: energy inequality 1(2)}).

Each term $I_{j}$ on the right-hand side of (\ref{eq: energy inequality 1(2)})
can be dominated as the following. The first three terms $I_{1}$,
$I_{2}$ and $I_{3}$ can be handled exactly as in B. Davies \cite{Davies-1987}
and D. Stroock \cite{Stroock-1988}. Recall that $\nabla\psi=\alpha$
is a constant vector. Hence 
\begin{equation}
I_{1}\leq\frac{\left|\alpha\right|^{2}}{\lambda}\Vert f_{t}\Vert_{2p}^{2p}.\label{eq:est1}
\end{equation}
While for $I_{2}$ and $I_{3}$, by completing squares we first rewrite
the terms of $I_{3}+I_{2}$ as following 
\begin{align*}
-I_{2}-I_{3} & =-\frac{2p-1}{p^{2}}\int_{\mathbb{R}^{n}}\left\langle \nabla f_{t}(x)^{p},a(x,t)\cdot\nabla f_{t}(x)^{p}\right\rangle dx\\
 & -2\frac{p-1}{p}\int_{\mathbb{R}^{n}}f_{t}(x)^{p}\left\langle \nabla f_{t}(x)^{p},a(x,t)\cdot\alpha\right\rangle dx\\
 & =-\frac{1}{p}\int_{\mathbb{R}^{n}}\left\langle \nabla f_{t}(x)^{p},a(x,t)\cdot\nabla f_{t}(x)^{p}\right\rangle dx+(p-1)\int_{\mathbb{R}^{n}}f_{t}(x)^{2p}\left\langle \alpha,a(x,t)\cdot\alpha\right\rangle dx\\
 & -\frac{p-1}{p^{2}}\int_{\mathbb{R}^{n}}\left\langle \left(\nabla f_{t}(x)^{p}-pf_{t}(x)^{p}\alpha\right),a(x,t)\cdot\left(\nabla f_{t}(x)^{p}-pf_{t}(x)^{p}\alpha\right)\right\rangle dx.
\end{align*}
The last term on the right-hand side is non-positive as $a(x,t)$
is positive definite, so by using inequalities
\[
\left\langle \nabla f_{t}(x)^{p},a(x,t)\cdot\nabla f_{t}(x)^{p}\right\rangle \geq\lambda\left|\nabla f_{t}(x)^{p}\right|^{2}
\]
and 
\[
\left\langle \alpha,a(x,t)\cdot\alpha\right\rangle \leq\frac{1}{\lambda}\left|\alpha\right|^{2}
\]
we deduce that 
\begin{equation}
-I_{2}-I_{3}\leq-\frac{\lambda}{p}\left\Vert \nabla f_{t}^{p}\right\Vert _{2}^{2}+\frac{p-1}{\lambda}\left|\alpha\right|^{2}\left\Vert f_{t}\right\Vert _{2p}^{2p}.\label{eq:est2}
\end{equation}

The main innovation in our proof is the handling of the skew-symmetric
part $I_{4}$ which does not appear in the symmetric case. The idea
is to apply the compensated estimate, Proposition \ref{the: compensated compactness 2}
, to obtain 
\begin{align}
\left|I_{4}\right| & =\left|2\int_{\mathbb{R}^{n}}f_{t}(x)^{p}\left\langle \nabla f_{t}(x)^{p},b(x,t)\cdot\alpha\right\rangle dx\right|\nonumber \\
 & \leq C\left\Vert b\right\Vert _{L^{\infty}\left(\textrm{BMO}\right)}\left|\alpha\right|\left\Vert f_{t}^{p}\right\Vert _{2}\left\Vert \nabla f_{t}^{p}\right\Vert _{2}\label{eq:est3}
\end{align}
where $C$ is a constant depending only on $n$. Therefore 
\begin{equation}
\left|I_{4}\right|\leq\frac{\lambda}{2p}\left\Vert \nabla f_{t}^{p}\right\Vert _{2}^{2}+C\left\Vert b\right\Vert _{L^{\infty}\left(\textrm{BMO}\right)}^{2}\frac{p}{\lambda}\left|\alpha\right|^{2}\left\Vert f_{t}\right\Vert _{2p}^{2p}.\label{eq:est4}
\end{equation}
Putting these estimates together we thus obtain (\ref{eq: energy inequality 1(2)}).\end{proof}

Now we can follow arguments in D. W. Stroock \cite{Stroock-1988}
to obtain the upper bound, yet again by using the special feature
of our elliptic operator. We contain the major steps only for completeness.

First we can prove the following by exactly the same argument in D.
W. Stroock \cite{Stroock-1988}.

\begin{lemma} There is a constant $C>0$ depending only on $n$ and
the $L_{t}^{\infty}(\textrm{BMO})$-norm of the skew-symmetric part
of $\left(A^{ij}(x,t)\right)$ such that 
\begin{equation}
\left\Vert \varGamma_{\tau,t}^{\psi}f\right\Vert _{\infty}\leq\frac{C}{(t-\tau)^{n/4}}e^{\frac{C\vert\alpha\vert^{2}(t-\tau)}{\lambda}}\Vert f\Vert_{2}\label{eq:dom01}
\end{equation}
and 
\begin{equation}
\left\Vert \varGamma_{\tau,t}^{\psi\dagger}f\right\Vert _{\infty}\leq\frac{C}{(t-\tau)^{n/4}}e^{\frac{C\vert\alpha\vert^{2}(t-\tau)}{\lambda}}\Vert f\Vert_{2}\label{eq:dom01-1}
\end{equation}
for every $f\in L^{2}\left(\mathbb{R}^{n}\right)$ , $0\leq\tau<t$
and $\alpha\in\mathbb{R}^{n}$, where $\psi(x)=\alpha\cdot x$. \end{lemma}

\begin{proof} We only need to prove (\ref{eq:dom01}) for the case
that $0=\tau<t$. The proof of (\ref{eq:dom01-1}) is similar, and
uses the inequality (\ref{eq: energy inequality 1(2)-1}) instead
and uses the fact that the constant appears in that inequality is
independent of any $T>0$.

To show (\ref{eq:dom01}), applying Nash's inequality (\ref{eq:nash1})
to the first term on the right-hand side of (\ref{eq: energy inequality 1(2)})
to deduce that 
\begin{equation}
\frac{d}{dt}\Vert f_{t}\Vert_{2p}\leq-\frac{\lambda}{2pC_{n}}\frac{\Vert f_{t}\Vert_{2p}^{1+4p/n}}{\Vert f_{t}\Vert_{p}^{4p/n}}+C_{B}\frac{\vert\alpha\vert^{2}}{\lambda}p\Vert f_{t}\Vert_{2p}\label{eq: energy inequality 2(2)}
\end{equation}
for every $p>1$. Let $u_{p}(t)=\Vert f_{t}\Vert_{2p}$ and $w_{p}(t)=\sup_{0\leq s\leq t}s^{n(p-2)/4p}u_{p/2}(s)$.
Then \eqref{eq: energy inequality 2(2)} may be written as 
\[
u_{p}'(t)\leq-\frac{\lambda}{2pC_{n}}\frac{t^{p-2}u_{p}(t)^{1+4p/n}}{(w_{p}(t))^{4p/n}}+C\frac{\vert\alpha\vert^{2}}{\lambda}pu_{p}(t)
\]
so that, according to Lemma \ref{lem: inequality 1(2)}, we have 
\begin{eqnarray*}
w_{2p}(t) & = & \sup_{0\leq s\leq t}s^{n(p-1)/4p}u_{p}(s)\\
 &  & \leq\sup_{0\leq s\leq t}(\frac{Kp^{2}}{\delta})^{\frac{n}{4p}}\exp(\frac{C\vert\alpha\vert^{2}\delta s}{p\lambda})w_{p}(s)\\
 & = & (\frac{Kp^{2}}{\delta})^{\frac{n}{4p}}\exp(\frac{C\vert\alpha\vert^{2}\delta t}{p\lambda})w_{p}(t).
\end{eqnarray*}
According to \eqref{eq: energy inequality 2(2)}, if take $p=1$,
we have 
\[
w_{2}(t)=\sup_{0\leq s\leq t}\Vert f_{s}\Vert_{2}\leq e^{C\vert\alpha\vert^{2}t/\lambda}\Vert f\Vert_{2}.
\]
Now we set $\delta=1$ and iterate it to get 
\[
w_{2^{m}}(t)\leq C\exp\left(\frac{C\vert\alpha\vert^{2}t}{\lambda}\right)w_{2}(t)\leq C\exp\left(\frac{C\vert\alpha\vert^{2}t}{\lambda}\right)\Vert f\Vert_{2}
\]
Letting $m\rightarrow\infty$, we therefore obtain that 
\[
\Vert f_{t}\Vert_{\infty}\leq\frac{C}{t^{n/4}}\exp\left(\frac{C\vert\alpha\vert^{2}t}{\lambda}\right)\Vert f\Vert_{2}
\]
which completes the proof. \end{proof}

\emph{Proof of the upper bound (\ref{eq:aron1-1}). }Let us use the
same notations as in the proof of the above lemma. By \eqref{eq:dom01-1}
and the fact that $\varGamma_{\tau,t}^{\psi\dagger}$ is the adjoint
operator of $\varGamma_{\tau,t}^{\psi}$, we have 
\[
\Vert f_{t}\Vert_{2}\leq\frac{C}{t^{n/4}}\exp\left(\frac{C\vert\alpha\vert^{2}t}{\lambda}\right)\Vert f\Vert_{1}.
\]
Since $\varGamma_{0,2t}^{\psi}=\varGamma_{t,2t}^{\psi}\circ\varGamma_{0,t}^{\psi}$,
we thus deduce that 
\[
\Vert f_{2t}\Vert_{\infty}\leq\frac{C}{t^{n/2}}\exp\left(\frac{C\vert\alpha\vert^{2}t}{\lambda}\right)\Vert f\Vert_{1},
\]
which is equivalent to 
\[
\varGamma(x,2t;\xi,0)\leq\frac{C}{t^{n/2}}\exp\left[\frac{C\vert\alpha\vert^{2}t}{\lambda}+\alpha\cdot(\xi-x)\right].
\]
Letting $\alpha=\frac{\lambda}{2Ct}(x-\xi)$ and adjusting $2t$ to
$t$ and $0$ to $\tau$, we therefore derive the upper bound 
\[
\varGamma(x,t;\xi,\tau)\leq\frac{C}{(t-\tau)^{\frac{n}{2}}}\exp\left(-\frac{\vert x-\xi\vert^{2}}{C(t-\tau)}\right)
\]
for any $t>\tau\geq0$. This completes the proof of the upper bound.

\subsection{Proof of the lower bound}

In this part, we prove the lower bound, 
\begin{equation}
\varGamma(x,t;\xi,\tau)\geq\frac{1}{C(t-\tau)^{\frac{n}{2}}}\exp\left[-\frac{C|x-\xi|^{2}}{t-\tau}\right]\label{eq:aron1-2}
\end{equation}
following the idea due to J. Nash \cite{Nash-1958}, where $C$ depends
only on $n$, $\lambda$ and $\left\Vert b\right\Vert _{L^{\infty}(\textrm{BMO})}$.

According to Nash's arguments, the lower bound is local in nature,
and follows easily from the following

\begin{lemma} There is a constant $C_{0}>0$ depending only on the
dimension $n$, $\lambda>0$ and the $L^{\infty}(\textrm{BMO})$ norm
of $\left(b^{ij}\right)$, such that

\begin{equation}
\int_{\mathbb{R}^{n}}\ln(\varGamma(x,1;\xi,0))\;\mu(d\xi)\geq-C_{0},\qquad\forall x\in B(0,2)\label{eq: lowerbound step BMO 1}
\end{equation}
and 
\begin{equation}
\varGamma(x,2;\xi,0)\geq e^{-2C_{0}},\qquad x,\xi\in B(0,2)\label{eq: lowerbound step BMO 2}
\end{equation}
where $\mu$ denotes the standard Gaussian measure in $\mathbb{R}^{n}$,
i.e. 
\[
\mu\left(d\xi\right)=\mu\left(\xi\right)d\xi\qquad\textrm{ where }\mu\left(\xi\right)=\frac{1}{(2\pi)^{n/2}}e^{-\frac{|\xi|^{2}}{2}}.
\]
\end{lemma}

\begin{proof} The proof follow the same ideas as in J. Nash, as explained
in D. W. Stroock \cite{Nash-1958}. We have to overcome difficulties
arising from the additional non-symmetric part $b(x,t)=\left(b^{ij}(x,t)\right)$.
The idea is to consider for any $x\in B(0,2)$ the following function
\[
G(t)=\int_{\mathbb{R}^{n}}\ln(\varGamma(x,1;\xi,1-t))\;\mu(d\xi)
\]
where $t\in(0,1]$ and $x\in B(0,2)$. Since $\left(A^{ij}\right)$
is uniformly elliptic with bounded derivatives, $\varGamma(x,t;\xi,\tau)$
is a probability density in $x$ (when others are fixed) and also
in $\xi$ (as other variables are fixed). Hence $\int_{\mathbb{R}^{n}}\varGamma(x,t;\xi,0)\;d\xi=1$
for every $t\in(0,1]$. According to Jensen's inequality $G(t)\leq0$
and what we want to show is that $G(1)$ is bounded from below uniformly
in $x\in B(0,2)$. To this end we consider the derivative of $G$.
By a simple calculation with integration by parts we obtain 
\begin{align}
G'(t) & =\int_{\mathbb{R}^{n}}\left\langle \xi,a(\xi,1-t)\cdot\nabla_{\xi}\ln\varGamma(x,t;\xi,1-t)\right\rangle \mu(d\xi)\nonumber \\
 & +\int_{\mathbb{R}^{n}}\left\langle \nabla_{\xi}\ln\varGamma(x,1;\xi,1-t),a(\xi,1-t)\cdot\nabla_{\xi}\ln\varGamma(x,t;\xi,1-t)\right\rangle \mu(d\xi)\label{eq:gd1.1}\\
 & +\frac{1}{\delta}\int_{\mathbb{R}^{n}}\left\langle \nabla\mu\left(\xi\right)^{\delta},b(\xi,1-t)\cdot\nabla_{\xi}\left(\mu(\xi)^{1-\delta}\ln\varGamma(x,t;\xi,1-t)\right)\right\rangle d\xi\nonumber 
\end{align}
for any $\delta\in(0,1)$. We have used the facts that $\nabla\ln\mu\left(\xi\right)=-2\xi$,
the backward equation for the fundamental solution $\varGamma(x,t;\xi,\tau)$
and the following fact that 
\[
\left\langle \nabla\mu,b\cdot\nabla\ln\varGamma\right\rangle =\frac{1}{\delta}\left\langle \nabla\mu^{\delta},b\cdot\nabla\left(\mu^{1-\delta}\ln\varGamma\right)\right\rangle 
\]
as $b$ is skew-symmetric. Using the Cauchy-Schwartz inequality and
the compensated compactness inequality (\ref{eq:c-c-e}) we deduce
that 
\begin{align*}
G'(t) & \geq-\frac{1}{\varepsilon}+\left(1-\varepsilon\right)\int_{\mathbb{R}^{n}}\left\langle \nabla_{\xi}\ln\varGamma(x,1;\xi,1-t),a(\xi,1-t)\cdot\nabla_{\xi}\ln\varGamma(x,t;\xi,1-t)\right\rangle \mu(d\xi)\\
 & -\frac{C}{\delta}\left\Vert b\right\Vert _{\textrm{BMO}}\left\Vert \nabla\mu^{\delta}\right\Vert _{2}\left\Vert \nabla_{\xi}\left(\mu^{1-\delta}\ln\varGamma(x,t;\cdot,1-t)\right)\right\Vert _{2}\\
 & \geq-\frac{1}{\varepsilon}+\left(1-\varepsilon\right)\lambda\left\Vert \nabla\ln\varGamma(x,1;\cdot,1-t)\right\Vert _{L^{2}\left(\mu\right)}\\
 & -\frac{C}{\delta}\left\Vert b\right\Vert _{\textrm{BMO}}\left\Vert \nabla\mu^{\delta}\right\Vert _{2}\left\Vert \nabla\left(\mu^{1-\delta}\ln\varGamma(x,t;\cdot,1-t)\right)\right\Vert _{2}
\end{align*}
for any $\varepsilon,\delta\in(0,1)$. Choose $\delta\in(0,\frac{1}{2})$.
Then 
\[
\left\Vert \nabla\mu^{\delta}\right\Vert _{2}=\delta\left\Vert \mu^{\delta}\nabla\ln\mu\right\Vert _{2}<\infty.
\]
Moreover, for $\delta\in(0,\frac{1}{2})$, 
\[
\sup_{\xi}\left|\mu(\xi)^{\frac{1}{2}-\delta}\nabla\ln\mu(\xi)\right|<\infty
\]
and 
\[
\sup_{\xi}\left|\mu(\xi)^{\frac{1}{2}-\delta}\right|<\infty,
\]
so that 
\begin{align*}
\left\Vert \nabla\left(\mu^{1-\delta}\ln\varGamma(x,t;\cdot,1-t)\right)\right\Vert _{2} & \leq\left\Vert (\nabla\mu^{1-\delta})\ln\varGamma(x,t;\cdot,1-t)\right\Vert _{2}\\
 & +\left\Vert \mu^{1-\delta}\nabla\ln\varGamma(x,t;\cdot,1-t)\right\Vert _{2}\\
 & =(1-\delta)\left\Vert (\mu^{\frac{1}{2}-\delta}\nabla\ln\mu)\ln\varGamma(x,t;\cdot,1-t)\right\Vert _{L^{2}\left(\mu\right)}\\
 & +\left\Vert \mu^{\frac{1}{2}-\delta}\nabla\ln\varGamma(x,t;\cdot,1-t)\right\Vert _{L^{2}(\mu)}\\
 & \leq C\left\{ \left\Vert \ln\varGamma(x,t;\cdot,1-t)\right\Vert _{L^{2}\left(\mu\right)}+\left\Vert \nabla\ln\varGamma(x,t;\cdot,1-t)\right\Vert _{L^{2}(\mu)}\right\} 
\end{align*}
for some constant $C$ depending only on $n$ and $\delta\in(0,\frac{1}{2})$.
By substituting this estimate into the inequality for $G'$ we obtain
\begin{align*}
G'(t) & \geq-\frac{1}{\varepsilon}+\left(1-\varepsilon\right)\lambda\left\Vert \nabla\ln\varGamma(x,1;\cdot,1-t)\right\Vert _{L^{2}\left(\mu\right)}^{2}\\
 & -C\left\Vert b\right\Vert _{\textrm{BMO}}\left\{ \left\Vert \ln\varGamma(x,t;\cdot,1-t)\right\Vert _{L^{2}\left(\mu\right)}+\left\Vert \nabla\ln\varGamma(x,t;\cdot,1-t)\right\Vert _{L^{2}(\mu)}\right\} \\
 & \geq-\frac{1}{\varepsilon}-\frac{1}{4\lambda^{2}\varepsilon^{2}}C^{2}\left\Vert b\right\Vert _{\textrm{BMO}}^{2}+\left(1-2\varepsilon\right)\lambda\left\Vert \nabla\ln\varGamma(x,1;\cdot,1-t)\right\Vert _{L^{2}\left(\mu\right)}^{2}\\
 & -C\left\Vert b\right\Vert _{\textrm{BMO}}\left\Vert \ln\varGamma(x,t;\cdot,1-t)\right\Vert _{L^{2}\left(\mu\right)}
\end{align*}
for any $\varepsilon\in(0,\frac{1}{2})$. By choosing $\varepsilon=1/3$,
we thus have the following differential inequality: 
\begin{eqnarray}
 & G'(t)\geq-C+\frac{\lambda}{3}\left\Vert \nabla\ln\varGamma(x,1;\cdot,1-t)\right\Vert _{L^{2}\left(\mu\right)}^{2}-C\left\Vert \ln\varGamma(x,t;\cdot,1-t)\right\Vert _{L^{2}\left(\mu\right)}\label{eq:geq0.1}
\end{eqnarray}
for all $t\in(0,1)$, for some constant $C>0$ depending only on $n$
and the $L^{\infty}\left(\textrm{BMO}\right)$ norm of the skew-symmetric
part $b(x,t)$.

The remaining arguments of the proof are more or less the same as
D. W. Stroock \cite{Stroock-1988}. Firstly, by the Poincaré inequality
for the Gaussian measure, we obtain 
\[
J(x,t)\equiv\left\Vert \ln\varGamma(x,1;\cdot,1-t)-G(t)\right\Vert _{L^{2}\left(\mu\right)}^{2}\leq2\left\Vert \nabla\ln\varGamma(x,1;\cdot,1-t)\right\Vert _{L^{2}\left(\mu\right)}^{2}.
\]
On the other hand, since $G(t)<0$, we have 
\begin{align*}
J(x,t) & =\left\Vert \ln\varGamma(x,1;\cdot,1-t)-G(t)\right\Vert _{L^{2}\left(\mu\right)}^{2}\\
 & =\int_{\mathbb{R}^{n}}\left(\ln\varGamma(x,1;\xi,1-t)-G(t)\right)^{2}\mu(d\xi)\\
 & \geq\int_{\left\{ \ln\varGamma(x,1;\xi,1-t)\geq-K\right\} }\left(\ln\varGamma(x,1;\xi,1-t)-G(t)\right)^{2}\mu(d\xi)\\
 & \geq\int_{\left\{ \ln\varGamma(x,1;\xi,1-t)\geq-K\right\} }\left(\ln\varGamma(x,1;\xi,1-t)+K-G(t)-K\right)^{2}\mu(d\xi)\\
 & \geq\frac{1}{2}\int_{\left\{ \ln\varGamma(x,1;\xi,1-t)\geq-K\right\} }\left(\ln\varGamma(x,1;\xi,1-t)+K-G(t)\right)^{2}\mu(d\xi)-K^{2}\\
 & \geq\frac{1}{2}\int_{\left\{ \ln\varGamma(x,1;\xi,1-t)\geq-K\right\} }G(t)^{2}\mu(d\xi)-K^{2}\\
 & =\frac{1}{2}G(t)^{2}\mu\left\{ \xi\in\mathbb{R}^{n}:\ln\varGamma(x,1;\xi,1-t)\geq-K\right\} -K^{2}.
\end{align*}
According to the upper bound 
\[
\varGamma(x,1;\xi,1-t)\leq\frac{C}{t^{n/2}}\exp\left[-\frac{|x-\xi|^{2}}{Ct}\right],
\]
we have 
\[
\ln\varGamma(x,1;\xi,1-t)\leq\ln C-\frac{n}{2}\ln t-\frac{|x-\xi|^{2}}{Ct}
\]
for $x\in B(0,2)$ and $t\in(\frac{1}{2},1)$. Hence
\begin{align*}
\int_{\left\{ |\xi|>r\right\} }\varGamma(x,1;\xi,1-t)d\xi & \leq\int_{\left\{ |\xi|>r\right\} }\frac{C}{t^{n/2}}\exp\left[-\frac{|x-\xi|^{2}}{Ct}\right]d\xi\\
 & =C_{1}\mu\left[\left|\sqrt{\frac{C}{2}t}\xi+x\right|>r\right]\\
 & \leq C_{1}\mu\left[\left|\xi\right|>\frac{r-2}{\sqrt{\frac{C}{2}t}}\right]\leq C_{1}\mu\left[\left|\xi\right|>\frac{r-2}{\sqrt{\frac{C}{2}}}\right]
\end{align*}
and therefore, there is a positive number $R$ depending on $C$ such
that for any $r>R$ 
\[
\int_{\left\{ |\xi|>r\right\} }\varGamma(x,1;\xi,1-t)\;d\xi<\frac{1}{4},\qquad\mbox{ for all }t\in(0,1],x\in B(0,2).
\]
Thus for any $t\in[\frac{1}{2},1]$, there is $M$ such that $\varGamma(x,1;\xi,1-t)\leq M$,
so that 
\begin{align*}
\frac{3}{4} & \leq\int_{B(0,r)}\varGamma(x,1;\xi,1-t)\;d\xi\leq\vert B(0,r)\vert e^{-K}\\
 & +(2\pi)^{n/2}Me^{r^{2}/2}\mu\left\{ \xi\in\mathbb{R}^{n}:\ln\varGamma(x,1;\xi,1-t)\geq-K\right\} .
\end{align*}
Choose $K>0$ such that $|B(0,r)|e^{-K}=\frac{1}{4}$, to obtain 
\begin{align*}
\mu\left\{ \xi\in\mathbb{R}^{n}:\ln\varGamma(x,1;\xi,1-t)\geq-K\right\}  & \geq\frac{1}{2(2\pi)^{n/2}Me^{r^{2}/2}}\\
 & \equiv\kappa(r)>0.
\end{align*}
By using this estimate we deduce that 
\begin{align*}
G'(t) & \geq-C+\frac{\lambda}{6}J(x,t)-C\left\Vert \ln\varGamma(x,t;\cdot,1-t)\right\Vert _{L^{2}\left(\mu\right)}\\
 & \geq-C+\frac{\lambda}{6}J(x,t)-C\left[\sqrt{J(x,t)}+|G(t)|\right]\\
 & \geq-C(\varepsilon,\lambda)+\frac{\lambda}{12}J(x,t)-\varepsilon G(t)^{2}\\
 & \geq-C(\varepsilon,\lambda)+\frac{\lambda}{24}G(t)^{2}\mu\left\{ \xi\in\mathbb{R}^{n}:\ln\varGamma(x,1;\xi,1-t)\geq-K\right\} -\frac{\lambda}{12}K^{2}\\
 & \geq-C(\varepsilon,K,\lambda)+\left(\frac{\lambda}{12}\kappa(r)-\varepsilon\right)G(t)^{2}\\
\end{align*}
for $\epsilon>0$ such that $\frac{\lambda}{12}\kappa(r)-\varepsilon>0$.
Now we obtain

\begin{equation}
G'(t)\geq-C_{1}+C_{2}G(t)^{2}\label{eq:df0.1}
\end{equation}
for any $t\in[\frac{1}{2},1]$, where $C_{1}>0$, $C_{2}\in(0,1]$.
The previous inequality (\ref{eq:df0.1}) may be written as 
\[
G'(t)\geq C_{2}\left(G-\sqrt{\frac{C_{1}}{C_{2}}}\right)\left(G+\sqrt{\frac{C_{1}}{C_{2}}}\right)
\]
together with the fact that $G<0$, it follows that 
\begin{equation}
G(1)\geq\min\left\{ -C_{1}-2\sqrt{\frac{C_{1}}{C_{2}}},-\frac{8}{3C_{2}}\right\} =-C_{0}.\label{eq:str0.1}
\end{equation}

The lower bound in (\ref{eq: lowerbound step BMO 1}) follows from
the Chapman-Kolmogrov equation and Jensen's inequality. In fact 
\begin{align*}
\ln\varGamma(x,2;\xi,0) & =\ln\left(\int_{\mathbb{R}^{n}}\varGamma\left(x,2;z,1\right)\varGamma(z,1;\xi,0)dz\right)\\
 & =\ln\left(\int_{\mathbb{R}^{n}}(2\pi)^{n/2}e^{|z|^{2}/2}\varGamma\left(x,2;z,1\right)\varGamma(z,1;\xi,0)\mu(dz)\right)\\
 & \geq\ln\left(\int_{\mathbb{R}^{n}}\varGamma\left(x,2;z,1\right)\varGamma(z,1;\xi,0)\mu(dz)\right)\\
 & \geq\int_{\mathbb{R}^{n}}\ln\varGamma\left(x,2;z,1\right)\mu(dz)+\int_{\mathbb{R}^{n}}\ln\varGamma(z,1;\xi,0)\mu(dz)\\
 & \geq-2C_{0}
\end{align*}
which yields (\ref{eq: lowerbound step BMO 2}). \end{proof}

\emph{Proof of the lower bound (\ref{eq:aron1-2}).} By using scaling
invariant properties, i.e. for any $r>0$ and $z\in\mathbb{R}^{n}$,
\begin{equation}
\varGamma(rx+z,r^{2}t;r\xi+z,0)=r^{-n}\varGamma^{A_{r,z}}(x,t;\xi,0)
\end{equation}
where $A_{r,z}(x,t)=A(rx+z,r^{2}t)$ and $\varGamma^{A}$ is the fundamental
solution associated with $A$. The transformation $A\rightarrow A^{r,z}$
preserves the elliptic constant $\lambda$ and more importantly the
$L^{\infty}(\textrm{BMO})$ norms, so we may apply (\ref{eq: lowerbound step BMO 2})
to $\varGamma^{A_{r,z}}$ to deduce that 
\begin{equation}
\varGamma(x,2t;\xi,0)\geq\frac{e^{-2A}}{t^{n/2}},\quad\vert\xi-x\vert<4t^{\frac{1}{2}}.
\end{equation}
Together with the same chain argument as in D. W. Stroock \cite{Stroock-1988}
by using the Chapman-Kolmogrov equation, we obtain the lower bound
accordingly.

\section{Weak solutions for non-symmetric parabolic equations}

Let us consider the parabolic equation 
\begin{equation}
\sum_{i,j=1}^{n}\frac{\partial}{\partial x^{i}}\left[A^{ij}(x,t)\frac{\partial}{\partial x^{j}}u(x,t)\right]-\frac{\partial}{\partial t}u(x,t)=0,\label{eq:non-01}
\end{equation}
where $A^{ij}=a^{ij}+b^{ij}$, $\left(a^{ij}\right)$ is symmetric
satisfying the uniform elliptic condition that $\lambda\leq(a^{ij})\leq\lambda^{-1}$
in the matrix sense, $\left(b^{ij}\right)$ is skew-symmetric. We
only assume that $A^{ij}$ are Borel measurable in $(x,t)$, and $b^{ij}(x,t)$
belong to the BMO space for every $t\geq0$, such that the BMO norms
$t\rightarrow\left\Vert b(\cdot,t)\right\Vert _{\textrm{BMO}}$ is
bounded, whose supremum norm is denoted by $\left\Vert b\right\Vert _{L^{\infty}(\textrm{BMO})}$,
as before.

\subsection{Weak solutions}

Let us consider Cauchy's initial and Dirichlet boundary problem associated
with (\ref{eq:non-01}). Let $D\subset\mathbb{R}^{n}$ be an open
subset with a smooth boundary. Given $\tau>0$, $u(x,t)$, which is
a locally integrable and Borel measurable function in $(x,t)\in D\times[\tau,\infty)$,
is a weak solution to the Dirichlet boundary problem of (\ref{eq:non-01})
with initial data $u(\cdot,\tau)=f$, if it holds that 
\begin{equation}
-\int_{\tau}^{\infty}\int_{D}\nabla\varphi(x,t)\cdot A(x,t)\nabla u(x,t)dxdt+\int_{\tau}^{\infty}\int_{D}u(x,s)\frac{\partial}{\partial s}\varphi(x,s)dsdx+\int_{D}f(x)\varphi(x,\tau)dx=0\label{eq:weak01}
\end{equation}
for any smooth function $\varphi(x,s)$ which has a compact support
in $D\times[\tau,\infty)$. To ensure that (\ref{eq:weak01}) is well
defined, we need to assume that $u\in L^{2}\left(\left[\tau,T\right];H^{1}(D)\right)$
and $u\in L^{\infty}\left(\left[\tau,T\right];L^{2}(D)\right)$, and
the initial data $f$ is locally integrable. Let $\varGamma^{D}(x,t;\xi,\tau)$
denote the corresponding fundamental solution.

\begin{lemma}\label{lemma-01}Suppose in addition that $A^{ij}$
are smooth, so that the fundamental solution $\varGamma(x,t;\xi,\tau)$
exists, is smooth, and satisfies Aronson estimate, and therefore 
\[
0<\varGamma^{D}(x,t;\xi,\tau)\leq\varGamma(x,t;\xi,\tau)\leq\frac{C}{(t-\tau)^{n/2}}\exp\left(-\frac{|x-\xi|^{2}}{C(t-\tau)}\right)
\]
for all $t>\tau$ and $x,\xi\in D$. If $f\in L^{2}(D)$, then $u(x,t)=\varGamma_{\tau,t}^{D}f(x)$
belongs to 
\[
C([\tau,\infty),L^{2}\left(D\right))\cap L^{\infty}([\tau,\infty),L^{2}\left(D\right))\cap L^{2}([\tau,\infty),H^{1}\left(D\right)).
\]
Moreover, we have the energy inequality
\begin{equation}
||u(\cdot,t)||_{2}^{2}+2\lambda\int_{\tau}^{t}\left\Vert \nabla u(\cdot,s)\right\Vert _{2}^{2}\leq||f||_{2}^{2}\label{eq:engg1}
\end{equation}
for all $t\geq\tau$, and $u(x,t)$ is also a weak solution to (\ref{eq:weak01}).\end{lemma}

\begin{proof}This is a well known result in the theory of parabolic
equations, and its proof is easy. Suppose $f$ is smooth with compact
support in $D$, then $u(x,t)=\varGamma_{\tau,t}^{D}f(x)$ is a classical
solution to the parabolic equation (\ref{eq:non-01}), so that 
\[
\sum_{i,j=1}^{n}\frac{\partial}{\partial x^{i}}\left(A^{ij}(x,t)\frac{\partial}{\partial x^{j}}u(x,t)\right)-\frac{\partial}{\partial t}u(x,t)=0
\]
for all $x\in D$ and $t\geq\tau$. It follows that 
\[
-\int_{D}\sum_{i,j=1}^{n}A^{ij}(x,t)\frac{\partial}{\partial x^{i}}u(x,t)\frac{\partial}{\partial x^{j}}u(x,t)dx-\frac{1}{2}\frac{\partial}{\partial t}\int_{D}u(x,t)^{2}dx=0
\]
for all $t>\tau$, and therefore we have the energy inequality 
\begin{equation}
||u(\cdot,t)||_{2}^{2}+2\lambda\int_{\tau}^{t}\left\Vert \nabla u(\cdot,s)\right\Vert _{2}^{2}ds\leq||f||_{2}^{2}\label{eq:eng04}
\end{equation}
for all $t>\tau$. From the energy inequality above, we deduce that
for every $f\in L^{2}(D)$, $u(x,t)=\varGamma_{\tau,t}^{D}f(x)$ belongs
to $L^{\infty}([\tau,\infty),L^{2}\left(D\right))$ and also to $L^{2}([\tau,\infty),H^{1}\left(D\right))$,
and the energy inequality remains true. Therefore for any $\varphi(x,t)$
which is smooth with compact support in $D\times[\tau,\infty)$, we
have 
\begin{equation}
\int_{\tau}^{t}\int_{D}u(x,s)\frac{\partial}{\partial s}\varphi(x,s)\;dxds-\int_{\tau}^{t}\int_{D}\nabla\varphi(x,s)\cdot A(x,s)\nabla u(x,s)dxds+\int_{D}f(x)\varphi(x,\tau)dx=0,\label{eq:w1}
\end{equation}
which is true for smooth $f$ with compact support, so that it remains
true for $f\in L^{2}(D)$ by the energy inequality above. That is
$u(x,t)=\varGamma_{\tau,t}^{D}f(x)$ is a weak solution with initial
data $f\in L^{2}(D)$.\end{proof}

Next we prove a uniqueness theorem for weak solutions. To this end
we need the following fact.

\begin{lemma}\label{lemma-w1}Let $T>\tau\geq0$. Then 
\begin{equation}
\left\Vert u(\cdot,T)\right\Vert _{2}^{2}-\left\Vert u(\cdot,0)\right\Vert _{2}^{2}=2\left\langle \frac{\partial}{\partial t}u,u\right\rangle _{L^{2}\left([\tau,T],H^{-1}\left(\mathbb{R}^{n}\right)\right),L^{2}\left([\tau,T],H^{1}\left(\mathbb{R}^{n}\right)\right)}\label{eq:w-01}
\end{equation}
for any $u\in L^{2}\left([\tau,T],H^{1}\left(\mathbb{R}^{n}\right)\right)$,
such that $\frac{\partial}{\partial t}u\in L^{2}\left([\tau,T],H^{-1}\left(\mathbb{R}^{n}\right)\right)$
and $u\in C\left(\left[\tau,T\right],L^{2}\left(\mathbb{R}^{n}\right)\right)$,
where $\left\langle \cdot,\cdot\right\rangle _{W^{\star},W}$ denotes
the pairing between a Banach space $W$ and its dual Banach space
$W^{*}$.\end{lemma}

\begin{proof}If $u(x,t)=\sum\varphi_{i}(x)\eta_{i}(t)$ where $\varphi_{i}\in H^{1}\left(\mathbb{R}^{n}\right)$
and $\eta_{i}$ are smooth with compact support in $[\tau,T]$, then
\[
\int_{\mathbb{R}^{n}}u(x,t)^{2}dx=\sum_{i,j}\eta_{i}(t)\eta_{j}(t)\int_{\mathbb{R}^{n}}\varphi_{i}(x)\varphi_{j}(x)dx
\]
so that 
\[
\frac{d}{dt}\int_{\mathbb{R}^{n}}u(x,t)^{2}dx=2\sum_{i,j}\eta'_{i}(t)\eta_{j}(t)\int_{\mathbb{R}^{n}}\varphi_{i}(x)\varphi_{j}(x)dx,
\]
and therefore 
\begin{align*}
\int_{\tau}^{T}\frac{d}{dt}\int_{\mathbb{R}^{n}}u(x,t)^{2}dxdt & =2\int_{\tau}^{T}\sum_{i,j}\eta'_{i}(t)\eta_{j}(t)\int_{\mathbb{R}^{n}}\varphi_{i}(x)\varphi_{j}(x)dxdt\\
 & =2\left\langle \frac{\partial}{\partial t}u,u\right\rangle _{L^{2}\left([\tau,T],H^{-1}\left(\mathbb{R}^{n}\right)\right),L^{2}\left([\tau,T],H^{1}\left(\mathbb{R}^{n}\right)\right)}.
\end{align*}
Hence 
\[
\left\Vert u(\cdot,T)\right\Vert _{2}^{2}-\left\Vert u(\cdot,\tau)\right\Vert _{2}^{2}=2\left\langle \frac{\partial}{\partial t}u,u\right\rangle _{L^{2}\left([\tau,T],H^{-1}\left(\mathbb{R}^{n}\right)\right),L^{2}\left([\tau,T],H^{1}\left(\mathbb{R}^{n}\right)\right)}.
\]
The above remains true for any $u\in L^{2}\left([\tau,T),H^{1}\left(\mathbb{R}^{n}\right)\right)$,
such that $\frac{\partial}{\partial t}u\in L^{2}\left([\tau,T),H^{-1}\left(\mathbb{R}^{n}\right)\right)$
and $u\in C\left(\left[\tau,T\right],L^{2}\left(\mathbb{R}^{n}\right)\right)$,
by the density property.\end{proof}

We are now in a position to show the following uniqueness theorem
for weak solutions.

\begin{theorem}\label{pro: solution regularity in time} Suppose
$A=a+b$ satisfies conditions stated at the beginning of the section,
i.e. $\lambda\leq\left(a^{ij}(x,t)\right)\leq\lambda^{-1}$ and $\left\Vert b\right\Vert _{L^{\infty}\left(\textrm{BMO}\right)}<\infty$.
Let $\tau\geq0$. Suppose $u(x,t)\in C([\tau,\infty),L^{2}(\mathbb{R}^{n}))\cap L^{2}([\tau,\infty),H^{1}(\mathbb{R}^{n}))$,
and satisfies 
\begin{equation}
\int_{\tau}^{\infty}\int_{\mathbb{R}^{n}}u(x,t)\frac{\partial}{\partial t}\varphi(x,t)\;dxdt=\int_{\tau}^{\infty}\int_{\mathbb{R}^{n}}\nabla\varphi(x,t)\cdot A(x,t)\nabla u(x,t)\;dxdt\label{eq: weak solution}
\end{equation}
for any $\varphi(x,t)$ is smooth with compact support in $\mathbb{R}^{n}\times\left(\tau,\infty\right)$,
then 
\begin{equation}
\frac{\partial u}{\partial t}\in L^{2}\left([\tau,\infty),H^{-1}(\mathbb{R}^{n})\right).\label{eq:weak67}
\end{equation}
Hence the following energy inequality holds: 
\begin{equation}
\left\Vert u(\cdot,T)\right\Vert _{2}^{2}+2\lambda\int_{\tau}^{T}\int_{\mathbb{R}^{n}}|\nabla u(x,t)|^{2}\;dxdt\leq\left\Vert u(\cdot,\tau)\right\Vert _{2}^{2}\label{eq:ener-01}
\end{equation}
for every $T>\tau$, and the uniqueness of weak solutions holds for
the initial problem of (\ref{eq:non-01}) in space $C([\tau,\infty),L^{2}(\mathbb{R}^{n}))\cap L^{2}([\tau,\infty),H^{1}(\mathbb{R}^{n})).$\end{theorem}

\begin{proof} Consider the linear functional 
\[
F_{t}(\psi)=\int_{\mathbb{R}^{n}}(A(x,t)\nabla u(x,t))\cdot\nabla\psi(x)\;dx
\]
for $\psi\in H^{1}(\mathbb{R}^{n})$. By using the compensated compactness
inequality(\ref{eq:c-c-e}) we have 
\begin{equation}
F_{t}(\psi)\leq\left(\Vert a\Vert_{L^{\infty}\left(\mathbb{R}^{n}\times[0,\infty)\right)}+C\left\Vert b\right\Vert _{L^{\infty}(\textrm{BMO})}\right)\Vert\nabla u(\cdot,t)\Vert_{L^{2}(\mathbb{R}^{n})}\Vert\nabla\psi\Vert_{L^{2}(\mathbb{R}^{n})}
\end{equation}
for any $\psi\in H^{1}(\mathbb{R}^{n})$. Hence by Riesz' representation
theorem, there exists a unique $w(\cdot,t)\in H^{1}(\mathbb{R}^{n})$
for every $t$ such that 
\begin{equation}
F_{t}(\psi)=\int_{\mathbb{R}^{n}}\left(\nabla w(x,t)\cdot\nabla\psi(x)+w(x,t)\psi(x)\right)\;dx,\label{eq:eq69}
\end{equation}
where 
\[
\left\Vert w(\cdot,t)\right\Vert _{H^{1}(\mathbb{R}^{n})}\leq\left(\Vert a\Vert_{L^{\infty}\left(\mathbb{R}^{n}\times[0,\infty)\right)}+C\left\Vert b\right\Vert _{L^{\infty}(\textrm{BMO})}\right)\Vert\nabla u(\cdot,t)\Vert_{L^{2}(\mathbb{R}^{n})}
\]
which implies that $w\in L^{2}([\tau,\infty),H^{1}(\mathbb{R}^{n}))$.

In terms of $w(x,t)$, \eqref{eq: weak solution} becomes 
\begin{equation}
\int_{\tau}^{T}\int_{\mathbb{R}^{n}}u(x,t)\varphi(x)\eta'(t)\;dxdt=\int_{\tau}^{T}\int_{\mathbb{R}^{n}}\left(\nabla w(x,t)\cdot\nabla\varphi(x)+w(x,t)\varphi(x)\right)\eta(t)\;dxdt\label{eq:eq70}
\end{equation}
for any $\eta\in C_{0}^{\infty}((\tau,T))$ and $\varphi\in C_{0}^{\infty}(\mathbb{R}^{n})$,
which can be written as 
\[
\int_{\tau}^{T}\left\langle u(\cdot,t),\varphi\right\rangle _{L^{2}}\eta'(t)dt=\int_{\tau}^{T}\left\langle w(\cdot,t),\varphi\right\rangle _{H^{1}}\eta(t)dt
\]
and can be extended to any $\varphi\in H^{1}(\mathbb{R}^{n})$. Since
\begin{align*}
\left|\int_{\tau}^{T}\left\langle w(\cdot,t),\varphi\right\rangle _{H^{1}}\eta(t)dt\right| & \leq\int_{\tau}^{T}\left\Vert w(\cdot,t)\right\Vert _{H^{1}}\left\Vert \varphi\right\Vert _{H^{1}}\eta(t)dt\\
 & =\left\Vert \varphi\right\Vert _{H^{1}}\int_{\tau}^{T}\left\Vert w(\cdot,t)\right\Vert _{H^{1}}\eta(t)dt\\
 & \leq\left\Vert \varphi\right\Vert _{H^{1}}\sqrt{\int_{\tau}^{T}\left\Vert w(\cdot,t)\right\Vert _{H^{1}}^{2}dt}\left\Vert \eta\right\Vert _{L^{2}\left([\tau,T)\right)},
\end{align*}
we obtain 
\[
\left|\int_{\tau}^{T}\left\langle u(\cdot,t),\varphi\right\rangle _{L^{2}}\eta'(t)dt\right|\leq\left\Vert \varphi\right\Vert _{H^{1}}\sqrt{\int_{\tau}^{T}\left\Vert w(\cdot,t)\right\Vert _{H^{1}}^{2}dt}\left\Vert \eta\right\Vert _{L^{2}\left([\tau,T)\right)}
\]
which implies that 
\[
\frac{d}{dt}\left\langle u(\cdot,t),\varphi\right\rangle _{L^{2}}\in L^{2}\left([\tau,T]\right)
\]
for every $\varphi\in H^{1}\left(\mathbb{R}^{n}\right)$. Moreover,
according to F. Riesz' representation 
\[
\left\Vert \frac{d}{dt}\left\langle u(\cdot,t),\varphi\right\rangle _{L^{2}}\right\Vert _{L^{2}[\tau,T]}\leq\left\Vert \varphi\right\Vert _{H^{1}}\sqrt{\int_{\tau}^{\infty}\left\Vert w(\cdot,t)\right\Vert _{H^{1}}^{2}dt}
\]
for any $\varphi\in H^{1}\left(\mathbb{R}^{n}\right)$. Therefore,
there is $\frac{\partial}{\partial t}u\in L^{2}\left([\tau,T],H^{-1}\left(\mathbb{R}^{n}\right)\right)$
such that 
\[
\int_{\tau}^{T}\left\langle \frac{\partial}{\partial t}u(\cdot,t),\varphi\right\rangle _{H^{-1},H^{1}}\eta(t)dt=-\int_{\tau}^{T}\left\langle u(\cdot,t),\varphi\right\rangle _{L^{2}}\eta'(t)dt
\]
for every $\varphi\in H^{1}\left(\mathbb{R}^{n}\right)$ and $\eta\in C_{0}^{\infty}\left(\tau,T\right)$.
The above equality can be written as 
\begin{align*}
\left\langle \frac{\partial}{\partial t}u,\varphi\otimes\eta\right\rangle  & =-\int_{\tau}^{T}\int_{\mathbb{R}^{n}}u(x,t)\frac{\partial}{\partial t}\left(\varphi(x)\eta(t)\right)dxdt\\
 & =-\int_{\tau}^{T}\int_{\mathbb{R}^{n}}\nabla(\varphi(x)\eta(t))\cdot A(x,t)\nabla u(x,t)\;dxdt
\end{align*}
where and in the remaining part of the proof, for simplicity, we use
$\left\langle \cdot,\cdot\right\rangle $ to denote the pairing between
$L^{2}\left([\tau,T),H^{1}\left(\mathbb{R}^{n}\right)\right)$ and
its dual space $L^{2}\left([\tau,T],H^{-1}\left(\mathbb{R}^{n}\right)\right)$.
Since 
\[
\textrm{span}\left\{ \varphi\otimes\eta:\varphi\in H^{1}\left(\mathbb{R}^{n}\right)\textrm{ and }\eta\in C_{0}^{\infty}\left(\tau,T\right)\right\} 
\]
is dense in $L^{2}\left([\tau,T],H^{1}\left(\mathbb{R}^{n}\right)\right)$,
we have 
\[
\left\langle \frac{\partial}{\partial t}u,\psi\right\rangle =-\int_{\tau}^{T}\int_{\mathbb{R}^{n}}\nabla\psi(x,t)\cdot A(x,t)\nabla u(x,t)\;dxdt
\]
for any $\psi\in L^{2}\left([\tau,T],H^{1}\left(\mathbb{R}^{n}\right)\right)$.
In particular, 
\begin{align*}
\left\langle \frac{\partial}{\partial t}u,u\right\rangle  & =-\int_{\tau}^{T}\int_{\mathbb{R}^{n}}\nabla u(x,t)\cdot A(x,t)\nabla u(x,t)\;dxdt\\
 & \leq-\lambda\int_{\tau}^{T}\int_{\mathbb{R}^{n}}|\nabla u(x,t)|^{2}dxdt.
\end{align*}
Now, by combining with Lemma \ref{lemma-w1}, we deduce that 
\[
\left\Vert u(\cdot,T)\right\Vert _{2}^{2}-\left\Vert u(\cdot,0)\right\Vert _{2}^{2}=2\left\langle \frac{\partial}{\partial t}u,u\right\rangle \leq-2\lambda\int_{\tau}^{T}\int_{\mathbb{R}^{n}}|\nabla u(x,t)|^{2}dxdt
\]
which in turn yields the energy inequality (\ref{eq:ener-01}). The
other conclusions of the theorem follow easily.\end{proof}

\begin{remark} 1) The same results has been proved by H. Osada \cite{Osada-1987}
when $A^{ij}$ are uniformly bounded. In fact the result is classical
if $A^{ij}$ are bounded (see for example Theorem 6.2. on page 102,
G. M. Lieberman \cite{Libeberman-book--1996}).

2) $u\in L^{2}([\tau,\infty),H^{1}(\mathbb{R}^{n}))$ and $\frac{\partial u}{\partial t}\in L^{2}([\tau,\infty),H^{-1}(\mathbb{R}^{n}))$
actually implies that $u\in C([\tau,\infty),L^{2}(\mathbb{R}^{n}))$
with possibly modification on a measure zero subset of $[\tau,\infty)$.
Therefore we have proved the uniqueness of weak solution $u$ in space
$L^{\infty}([\tau,\infty),L^{2}(\mathbb{R}^{n}))\cap L^{2}([\tau,\infty),H^{1}(\mathbb{R}^{n}))$.
\end{remark}

We are now in a position to state and to prove the following theorem.

\begin{theorem}\label{the: wellposeness theorem} Suppose $\left(A^{ij}\right)=\left(a^{ij}\right)+\left(b^{ij}\right)$,
where $a$ and $b$ are symmetric and skew-symmetric parts of $A$
respectively, is uniformly elliptic: $\lambda\leq a(x,t)\leq\lambda^{-1}$
in matrix sense for some constant $\lambda>0$, and $\left\Vert b\right\Vert _{L^{\infty}(\textrm{BMO})}<\infty$.
Then there is a unique positive function $\varGamma(x,t;\xi,\tau)$
defined for $t>\tau\geq0$ and $x,\xi\in\mathbb{R}^{n}$, which possesses
the following properties.

1) $\varGamma$ is a Markov transition density: $\varGamma(x,t;\xi,\tau)>0$,
\[
\int_{\mathbb{R}^{n}}\varGamma(x,t;\xi,\tau)d\xi=1\quad\textrm{and }\quad\int_{\mathbb{R}^{n}}\varGamma(x,t;\xi,\tau)dx=1
\]
for any $t>\tau\geq0$, and 
\[
\varGamma(x,t;\xi,\tau)=\int_{\mathbb{R}^{n}}\varGamma(x,t;z,s)\varGamma(z,s;\xi,\tau)dz
\]
for any $t>s>\tau\geq0$.

2) There is a constant $M>0$ depending only on $n$, $\lambda$ and
$\left\Vert b\right\Vert _{L^{\infty}(\textrm{BMO})}$ such that 
\[
\frac{1}{M(t-\tau)^{n/2}}\exp\left(-\frac{M|x-\xi|^{2}}{t}\right)\leq\varGamma(x,t;\xi,\tau)\leq\frac{M}{(t-\tau)^{n/2}}\exp\left(-\frac{|x-\xi|^{2}}{Mt}\right)
\]
for all $t>\tau$.

3) For every $f\in L^{2}\left(\mathbb{R}^{n}\right)$, $u(x,t)=\int_{\mathbb{R}^{n}}f(\xi)\varGamma(x,t;\xi,\tau)d\xi$
(for any $t\geq\tau$) is the unique weak solution with initial data
$f$, which belongs to 
\[
C([\tau,\infty),L^{2}(\mathbb{R}^{n}))\cap L^{\infty}([\tau,\infty),L^{2}(\mathbb{R}^{n}))\cap L^{2}([\tau,\infty),H^{1}(\mathbb{R}^{n})).
\]
 \end{theorem}

\begin{proof} Since $b\in L^{\infty}(\textrm{BMO}(\mathbb{R}^{n}))$,
we can choose a $\epsilon>0$ such that 
\[
\Vert\epsilon\log(\vert x\vert)\Vert_{\textrm{BMO}}\leq\Vert b\Vert_{L^{\infty}(\textrm{BMO})}.
\]
Define 
\begin{equation}
U^{(m)}(x)=(-\epsilon\log(\vert x\vert)+m)\wedge m\vee0,\qquad L^{(m)}(x)=(\epsilon\log(\vert x\vert)-m)\wedge0\vee(-m)
\end{equation}
which are compactly supported BMO functions with 
\[
\Vert U^{(m)}\Vert_{\textrm{BMO}}=\Vert L^{(m)}\Vert_{\textrm{BMO}}\leq C\Vert b\Vert_{L^{\infty}(\textrm{BMO})},
\]
where constant $C>0$ depends only on the dimension $n$. Let 
\begin{equation}
b^{(m)}(x,t)=b(x,t)\wedge U^{(m)}(x)\vee L^{(m)}(x).
\end{equation}
By Proposition \ref{prop-bmo}, we can mollify it to define $b_{\frac{1}{m}}^{(m)}$.
Then, there is a $C$ independent of $b$ and $m$, such that $\Vert b_{\frac{1}{m}}^{(m)}\Vert_{L^{\infty}(\textrm{BMO})}\leq C\Vert b\Vert_{L^{\infty}(\textrm{BMO})}$.
Each $b_{\frac{1}{m}}^{(m)}$ is smooth with compact support and $b_{\frac{1}{m}}^{(m)}\rightarrow b$
in $L_{loc}^{p}([0,\infty)\times\mathbb{R}^{n})$ for any $1\leq p<\infty$.
For simplicity denote $b_{\frac{1}{m}}^{(m)}$ by $b_{m}$. Similarly
$a_{m}$ denotes the mollified one of $a$ for $m=1,2,\cdots$. $a_{m}(x,t)$
and $b_{m}(x,t)$ are smooth, bounded and have bounded derivatives
of all orders, and $a_{m}\rightarrow a$ and $b_{m}\rightarrow b$
in $L_{loc}^{p}([0,\infty)\times\mathbb{R}^{n})$ for every $p\in[1,\infty)$.

Now for each $A_{m}(x,t)=a_{m}(x,t)+b_{m}(x,t)$, $a_{m}$ is uniformly
elliptic with elliptic constant $2\lambda$ and 
\[
\Vert b_{m}\Vert_{L^{\infty}(\textrm{BMO})}\leq C\Vert b\Vert_{L^{\infty}(\textrm{BMO})}
\]
for some constant depending only on the dimension $n$, thus there
is a unique fundamental solution $\varGamma^{m}(x,t;\xi,\tau)$ which
satisfies the Aronson estimate with the same constant. According to
Theorem \ref{cor: continuity of transition probability}, $\varGamma^{m}(x,t;\xi,\tau)$
are Hölder continuous in any compact sub-set of $t>\tau\geq0$ and
$x,\xi\in\mathbb{R}^{n}$ with the same Hölder exponent and the same
Hölder constant for all $m=1,2,\cdots$. Therefore by the Arzela-Ascoli
Theorem, there is a sub-sequence of $\varGamma^{m}$, for simplicity
the sub-sequence is still denoted by $\varGamma^{m}$, which converges
locally uniformly to some $\varGamma(x,t;\xi,\tau)$ for $t>\tau\geq0$
and $x,\xi\in\mathbb{R}^{n}$. Clearly $\varGamma(x,t;\xi,\tau)$
still satisfies 1) and 2).

We now prove 3). By our construction, if $\tau>0$ and $f\in L^{2}(\mathbb{R}^{n})$,
\[
u^{m}(x,t)=\varGamma_{\tau,t}^{m}f(x)\rightarrow u(x,t)=\varGamma_{\tau,t}f(x)
\]
point-wisely. According to Lemma \ref{lemma-01}, $u^{m}$ (actually
$u^{m}(x,t)$ is Hölder continuous too in $t>\tau$ and $x$) is a
strong solution to the Cauchy problem of the parabolic equation associated
with the diffusion matrix $A_{m}$, so that the energy inequality
holds: 
\begin{equation}
||u^{m}(\cdot,t)||_{2}^{2}+\lambda\int_{\tau}^{t}\left\Vert \nabla u^{m}(\cdot,s)\right\Vert _{2}^{2}\leq||f||_{2}^{2},\label{eq:engg1-1}
\end{equation}
which implies that $\left\{ u^{m}\right\} $ is uniformly bounded
in $L^{2}([\tau,\infty),H^{1}(\mathbb{R}^{n}))$. Hence there is a
sub-sequence converges weakly, whose limit must be $u$, and $u$
also satisfies the energy inequality above. In particular 
\[
u\in C\left([\tau,\infty),L^{2}\left(\mathbb{R}^{n}\right)\right)\cap L^{\infty}([\tau,\infty),L^{2}(\mathbb{R}^{n}))\cap L^{2}([s,\infty),H^{1}(\mathbb{R}^{n})).
\]
For each $m$, we have 
\[
\int_{\tau}^{\infty}\int_{\mathbb{R}^{n}}u^{m}(x,t)\frac{\partial}{\partial t}\varphi(x,t)\;dxdt-\int_{\tau}^{\infty}\int_{\mathbb{R}^{n}}\nabla\varphi(x,t)\cdot A_{m}(x,t)\nabla u^{m}(\tau,x))\;dxdt=0
\]
for any $\varphi\in C_{0}^{\infty}((\tau,\infty)\times\mathbb{R}^{n})$.
Since $A_{m}\rightarrow A$ in $L_{loc}^{p}([s,\infty)\times\mathbb{R}^{n})$
for any $1\leq p<\infty$ and $u_{m}\rightarrow u$ weakly in $L^{2}([\tau,\infty),H^{1}(\mathbb{R}^{n}))$.
By letting $m\rightarrow\infty$ in the equality above, we obtain
that 
\[
\int_{\tau}^{\infty}\int_{\mathbb{R}^{n}}u(x,t)\frac{\partial}{\partial t}\varphi(x,t)\;dxdt-\int_{\tau}^{\infty}\int_{\mathbb{R}^{n}}\nabla\varphi(x,t)\cdot A(x,t)\nabla u(x,t))\;dxdt=0
\]
for any $\varphi\in C_{0}^{\infty}((\tau,\infty)\times\mathbb{R}^{n})$.
That is, $u$ is a weak solution to the Cauchy problem of the parabolic
equation (\ref{eq:non-01}) with initial data $f$. Now according
to Lemma \ref{pro: solution regularity in time}, $\frac{\partial u}{\partial t}\in L^{2}([\tau,\infty),H^{-1}(\mathbb{R}^{n}))$
and therefore, by Theorem \ref{the: wellposeness theorem}, 
\begin{equation}
||u(\cdot,t)||_{2}^{2}+\lambda\int_{\tau}^{t}\left\Vert \nabla u(\cdot,s)\right\Vert _{2}^{2}\leq||f||_{2}^{2}.\label{eq:engg1-1-1}
\end{equation}

The uniqueness of the fundamental solution $\varGamma$ follows from
the energy inequality easily. In fact, suppose there is another sub-sequence
of $\varGamma^{m}$ converges to $\tilde{\varGamma}$. Then $\tilde{u}(x,t)=\tilde{\varGamma_{\tau,t}}f(x)$
satisfies all the results above. Especially they both satisfy the
energy inequality. Therefore $w=u-\tilde{u}$ also satisfies the previous
integral equation, and by Theorem \ref{the: wellposeness theorem},
we deduce that 
\[
\int_{\mathbb{R}^{n}}w(x,t)^{2}\;dx+\lambda\int_{\tau}^{t}\int_{\mathbb{R}^{n}}\vert\nabla w(x,t)\vert^{2}\;dxdt\leq0
\]
for any $t>\tau$ and we have $w=0$. This implies $u=\tilde{u}$
and hence $\varGamma=\tilde{\varGamma}$. The proof is complete. \end{proof}

\end{document}